\RequirePackage[british]{babel}
\documentclass[a4paper,11pt]{amsart}
\usepackage{amssymb,amscd,amsfonts,amsbsy,nicefrac}
\usepackage{amsmath,amsthm,amssymb,amsfonts,enumitem}
\usepackage{mathrsfs}
\usepackage[pdftex]{graphicx,color}
\usepackage{mathrsfs}
\usepackage[a4paper,tmargin=3cm,bmargin=3cm,lmargin=3cm,rmargin=3cm]{geometry}

\newtheorem{thm}{Theorem}[section]
\newtheorem{lem}[thm]{Lemma}

\newtheorem{prop}[thm]{Proposition}
\newtheorem{defn}[thm]{Definition}

\def\R{{\mathbb R}}

\def\C{{\mathbb C}}

\def\2q{{\frac{2}{|B|}}}

\newcommand{\N}{\mathbb{N}}

\newcommand{\supp}{\mbox{supp}\,}

\renewcommand{\leq}{\leqslant}
\renewcommand{\geq}{\geqslant}

\def\bra#1{{\langle{#1}\rangle}}

\newcommand{\xxi}{\langle \xi\rangle}
\newcommand{\phase}{\varphi}
\newcommand{\inverse}{\check}
\definecolor{red}{RGB}{255,0,0}
\definecolor{RED}{RGB}{255,0,0}
\definecolor{blue}{RGB}{0,0,255}
\definecolor{BLUE}{RGB}{0,0,255}
\newcommand{\abs}[1]{\left|#1\right|}
\newcommand{\set}[1]{\left\{#1\right\}}
\newcommand{\brkt}[1]{\left(#1\right)}
\newcommand{\dd}{\mathrm{d}}
\newcommand{\ddd}{\mbox{\dj}}
\renewcommand{\d}{\partial}
\newcommand{\norm}[1]{\left\Vert#1\right\Vert}
\newcommand{\p}[1]{\langle{#1}\rangle}

\begin{document}

\title[Boundedness of bilinear Fourier integral operators] {On the boundedness of certain bilinear Fourier integral operators}
\author[S.~Rodr\'iguez-L\'opez]{Salvador Rodr\'iguez-L\'opez}
\address{Department of Mathematics, Uppsala University, Uppsala, SE 75106,  Sweden}
\email{salvador@math.uu.se}
\author[D.J.~Rule]{David J.\ Rule}
\address{Department of Mathematics and the Maxwell Institute of Mathematical Sciences, Heriot-Watt University, Colin Maclaurin Building, Edinburgh, EH14 4AS, United Kingdom}
\email{rule@uchicago.edu}%Please leave as uchicago address until I move to Link\"oping
\author[W.~Staubach]{Wolfgang Staubach}
\address{Department of Mathematics, Uppsala University, Uppsala, SE 75106,  Sweden}
\email{wulf@math.uu.se}
\date{\today}
\thanks{The first author has been partially supported by the Grant MTM2010-14946. The second author gratefully acknowledges support from CANPDE}

\begin{abstract}
We prove the global $L^2 \times L^2 \to L^1$ boundedness of bilinear Fourier integral operators with amplitudes in $S^0_{1,0} (n,2)$. To achieve this, we require that the phase function can be written as $(x,\xi,\eta) \mapsto \phase_1(x,\xi) + \phase_2(x,\eta)$ where each $\phase_j$ belongs to the class $\Phi^2$ and satisfies the strong non-degeneracy condition. This result extends that of R.~Coifman and Y.~Meyer regarding pseudodifferential operators to the case of Fourier integral operators.
\end{abstract}

\maketitle

\section{Introduction}

We study {\it{bilinear Fourier integral operators}} $T_{\sigma,\phase}$ defined to act on Schwartz functions $f$ and $g$ by the formula
\begin{equation*}\label{defn RS FIO}
T_{\sigma,\phase}(f,g)(x) = \iint \sigma(x,\xi,\eta)\widehat{f}(\xi)\widehat{g}(\eta)e^{i\phase(x,\xi,\eta)} \ddd\xi \ddd\eta,
\end{equation*}
where $\ddd$ is Lebesgue measure normalised by the factor $(2\pi)^{-n}$, so $\ddd\xi = (2\pi)^{-n}\dd\xi$ and $\ddd\eta = (2\pi)^{-n}\dd\eta$, and
\[
	\widehat{f}(\xi)=\int f(x)e^{-i x.\xi}\, \dd x
\]
is the Fourier transform of $f$.
The function $\sigma \colon \R^n\times\R^n\times\R^n \to \C$ is called the {\it{amplitude}} of $T_{\sigma,\phase}$ and $\phase\colon \R^n\times\R^n\times\R^n \to \R$ is the {\it{phase function}} or {\it{phase}}.

The study of bilinear Fourier integral operators was initiated by L.~Grafakos and M.~Peloso \cite{GP}, where the authors study the local boundedness of bilinear Fourier integral operators on Banach and quasi-Banach $L^p$ spaces.

Of interest in relation to this paper are their results concerning phases of the form
\begin{equation*}\label{defn phase}
\phase(x,\xi,\eta) = \phase_1(x,\xi) + \phase_2(x,\eta),
\end{equation*}
in which case we write $T_{\sigma,\phase} = T^{\phase_1,\phase_2}_\sigma$. They assume $\sigma \in S^m_{1,0}(n,2)$ (see Definition \ref{defn of hormander amplitudes} below) is compactly supported in the first variable, and the phase functions $\phase_j \in \mathcal{C}^\infty(\R^n \times (\R^n\setminus\{0\}))$ are homogeneous of degree 1 in the second variable and verify the {\it{non-degeneracy condition}}
\[
\left|\det\,\d^{2}_{x, \xi} \phase_{j}(x,\xi)\right|\neq 0
\]
for $j=1,2$. They then deduce that the bilinear operator $T^{\phase_1,\phase_2}_\sigma$ is bounded from $L^{q_1}\times L^{q_2} \to L^{r}$ with $\frac{1}{q_1} +\frac{1}{q_2} =\frac{1}{r}$ and $1< q_{1}, \, q_{2} <2,$ provided that the order
\[
m \leq -(n-1)\left(\left(\frac{1}{q_1}-\frac{1}{2}\right)+\left(\frac{1}{q_2}-\frac{1}{2}\right)\right).
\]

The definition of $S^m_{1,0}(n,2)$ is due to R.~Coifman and Y.~Meyer \cite{CM4}. However recalling the definition of the symbols of general linear pseudodifferential operators as in L. H\"ormander \cite{H1}, one can define a more general class of bilinear amplitudes, which has been extensively studied in A. B\'enyi, D. Maldonado, V. Naibo and R. Torres \cite{BMNT}.
\begin{defn}\label{defn of hormander amplitudes}
Let $m\in \mathbf{R}$, $0\leq \delta\leq 1$, $0\leq \rho\leq 1.$ A function $\sigma \in \mathcal{C}^{\infty}(\R^{n} \times \R^{n} \times\R^{n})$ belongs to the class $S^{m}_{\rho,\delta} (n,2)$, if for all multi-indices $\alpha, \beta$ and $\gamma$, there exist constants $C_{\alpha,\beta,\gamma}$ such that
   \begin{align*}
  |\partial_{\xi}^{\alpha}\partial_{\eta}^{\beta}\partial_{x}^{\gamma}\sigma(x,\xi,\eta)| \leq C_{\alpha,\beta,\gamma}   (1+|\xi| + |\eta|)^{m-\rho(\vert \alpha\vert + \vert \beta\vert) + \delta |\gamma|}.
   \end{align*}
\end{defn}

Now we digress slightly to recall the definition of {\it linear Fourier integral operators}, as this will provide the necessary setting for further background results and be useful in the proof of the main result. They are operators given by
\[
T_{\sigma}^{\phase}(f)(x) = \int \sigma(x,\xi)\widehat{f}(\xi)e^{i\phase(x,\xi)}\ddd\xi,
\]
for a Schwartz function $f$. The given function $\phase \colon \R^n\times\R^n \to \R$ is the phase of $T_{\sigma}^{\phase}$, and $\sigma \colon \R^n\times\R^n \to \C$ is its amplitude. In~\cite{DSFS}, D.~Dos Santos Ferreira and W.~Staubach made a systematic study of the regularity of the aforementioned linear Fourier integral operators with both smooth and rough amplitudes and demonstrated various boundedness results under appropriate conditions on the amplitudes and phases. One of the classes of amplitudes considered in~\cite{DSFS} was the H\"ormander class, which is defined as follows.
  \begin{defn}\label{defn of hormander symbols}
Let $m\in \R$, $0\leq \delta\leq 1$, $0\leq \rho\leq 1.$ A function $\sigma \in \mathcal{C}^{\infty}(\R^{n} \times\R^{n})$ belongs to the class $S^{m}_{\rho,\delta}(n,1)$, if for all multi-indices $\alpha$ and $\beta$, there exist constants $C_{\alpha,\beta}$ such that
   \begin{align*}
      |\partial_{\xi}^{\alpha}\partial_{x}^{\beta}\sigma(x,\xi)| \leq C_{\alpha,\beta} \langle\xi\rangle^{m-\rho\vert \alpha\vert + \delta |\beta|},
   \end{align*}
where the notation $\langle \cdot\rangle$ stands for $(1+|\cdot|^2)^{1/2}$.
\end{defn}
Here we would like to remark that the notation $S^{m}_{\rho,\delta}(n,1)$ is to emphasise the linearity of the operators associated to the amplitudes in $S^{m}_{\rho,\delta}(n,1)$, in contrast to the bilinear operators associated to $S^{m}_{\rho,\delta}(n,2)$ from Definition \ref{defn of hormander amplitudes}. Indeed $S^{m}_{\rho,\delta}(n,N)$ denotes the class of amplitudes corresponding to $N$-linear operators (see \cite{MRS3}).

The class of phase functions $\Phi^k$ introduced in~\cite{DSFS} will, in the case of $k=2$, play an important role in our study here. Therefore we recall the definition of these type of phases.
\begin{defn}\label{Phik phases}
A real-valued function $\phase \colon \R^n\times\R^n \to \R$ belongs to the class $\Phi^{k}$, if $\phase \in \mathcal{C}^{\infty}(\R^n \times\R^n \setminus \{0\})$, is positively homogeneous of degree $1$ in the second variable, and satisfies the following condition:
For any pair of multi-indices $\alpha$ and $\beta$, such that $|\alpha|+|\beta|\geq k$, there exists a positive constant $C_{\alpha, \beta}$ such that
   \begin{align*}
      \sup_{(x,\,\xi) \in \R^n \times\R^n \setminus \{0\}}  |\xi| ^{-1+\vert \alpha\vert}\vert \partial_{\xi}^{\alpha}\partial_{x}^{\beta}\phase(x,\xi)\vert
      \leq C_{\alpha,\beta}.
   \end{align*}
\end{defn}

It was shown in~\cite{DSFS} that it is necessary to assume the so called {\it{strong non-degeneracy}} condition on the phases, in order to guarantee global regularity. The strong non-degeneracy condition is defined as follows.

\begin{defn}[The strong non-degeneracy condition]\label{strong non-degeneracy}
A real valued phase $\phase\in \mathcal{C}^{\infty}(\R^n \times\R^n \setminus \{0\})$ satisfies the strong non-degeneracy condition, if there exists a positive constant $c$ such that
\[
\left|\det\,\d^{2}_{x, \xi} \phase_{j}(x,\xi)\right| \geq c,
\]
for all $(x,\,\xi) \in \R^n \times\R^n \setminus \{0\}$.
\end{defn}

We remark that phases in class $\Phi^2$ satisfying the strong non-degeneracy condition arise naturally in the study of hyperbolic partial differential equations, indeed a phase function closely related to that of the wave operator, namely $\phase(x,\xi)= |\xi|+ x\cdot\xi$ belongs to the class $\Phi^2$ and is strongly non-degenerate.

In the context of our current investigation, a useful result regarding the global boundedness of linear Fourier integral operators was established in Theorem 2.2.6 of \cite{DSFS}. This result, which is stated below, applies to operators with symbols belonging to the class $S^m_{\rho,\delta}(n,1)$.
\begin{thm} \label{DW226}
If $m = \min\{0,n(\rho-\delta)/2\}$, $0 \leq \rho \leq 1$, $0 \leq \delta < 1$, $\sigma \in S^m_{\rho,\delta}(n,1)$ and $\phase \in \Phi^2$ satisfies the strong non-degeneracy condition, then the operator $T_\sigma^\phase$ is bounded on $L^2(\R^n)$ and its norm is bounded by a constant depending only on $n$, $m$, $\rho$, $p$, $c$ in Definition \ref{strong non-degeneracy}, and a finite number of $C_{\alpha,\beta}$ appearing in Definitions \ref{Phik phases} and \ref{defn of hormander symbols}.
\end{thm}

In \cite{RS}, S. Rodr\'iguez-L\'opez and W. Staubach prove the boundedness of linear Fourier integral operators with amplitudes in the class $L^pS^m_\rho(n,1)$, first introduced by N.~Michalowski, D.~Rule and W.~Staubach in~\cite{MRS3}. Here we recall the definition of the class $L^pS^m_\rho(n,1)$.
\begin{defn} \label{def5}
Let $1 \leq p \leq \infty$, $m \in \R$ and $0 \leq \rho \leq 1$ be parameters. A symbol $\sigma \colon \R^n \times \R^{n} \to \C$ belongs to the class $L^p S^m_{\rho}(n,1)$ if for each multi-index $\alpha$ there exists a constant $C_\alpha$ such that
\begin{equation*} \label{symr}
\sup_{\xi \in \R^n}\langle \xi\rangle^{-m+\rho |
\alpha|}\|\partial_\xi^\alpha \sigma(\cdot,\xi)\|_{L^p(\R^n)} \leq C_\alpha.
\end{equation*}
\end{defn}
The motivation for introducing the class $L^pS^m_\rho(n,1)$ in \cite{MRS3} is that it proves useful in the study of bilinear operators. Indeed, the global regularity of bilinear Fourier integral operators with amplitudes that are neither compactly supported nor smooth in the first variable is studied in \cite{RS} in part by proving the following linear result (which is also of direct use to us in this paper).
\begin{thm} \label{SW}
Suppose that $2 \leq p \leq \infty$ and let $r = 2p/(p+2)$. Let $\sigma \in L^pS^m_\rho(n,1)$, $\phase \in \Phi^2$ satisfy the strong non-degeneracy condition and assume $m < n(\rho - 1)/2$. Then the operator $T_\sigma^\phase$ is bounded from $L^2(\R^n)$ to $L^r(\R^n)$ and its norm is bounded by a constant depending only on $n$, $m$, $\rho$, $p$, $c$ in Definition \ref{strong non-degeneracy}, and a finite number of $C_{\alpha,\beta}$ and $C_\alpha$ appearing in Definitions \ref{Phik phases} and \ref{def5}.
\end{thm}

And so, we return to our main subject of interest: bilinear Fourier integral operators. In \cite{RS} it was also shown that if $\phase_1,\phase_2\in \Phi^2$ verify the strong non-degeneracy condition, and $\sigma$ verifies the estimate
\begin{equation*}
\norm{\partial_{\xi}^{\alpha}\partial_{\eta}^{\beta} \sigma(\cdot,\xi,\eta)}_{L^{\infty}(\R^n)}\leq C_{\alpha,\beta} (1+|\xi|+|\eta|)^{m-|\alpha|-|\beta|},
\end{equation*}
then $T^{\phase_1,\phase_2}_\sigma$ is bounded from $L^{q_1} \times L^{q_2} \to L^{r}$ provided that $\frac{1}{r}= \frac{1}{q_1}+ \frac{1}{q_2}$, $1\leq q_1 ,\, q_2 \leq \infty$ and
\[
    m<-(n-1)\brkt{\abs{\frac{1}{q_1}-\frac{1}{2}}+\abs{\frac{1}{q_2}-\frac{1}{2}}}.
\]
Moreover in the case $m<0$, if the phases $\phase_{j}\in \mathcal{C}^{\infty}(\R^n \times \R^n)$ are inhomogeneous, strongly non-degenerate and verify the condition $ |\partial_{x}^{\alpha} \partial^{\beta}_{\xi} \phase_{j} (x, \xi)|\leq C_{j,\alpha,\beta}$ for $j=1,2$ and all multi-indices $\alpha $ and $\beta$ with $2\leq |\alpha|+|\beta|$,  then $T^{\phase_1,\phase_2}_\sigma$ is bounded from $L^2 \times L^2 \to L^{1}$.

The goal of this paper is to extend the results in \cite{GP} and \cite{RS} described above to the case of the end-point $m=0$, $q_1 =q_2 =2$ and $r=1$ for bilinear Fourier integral operators with strongly non-degenerate phases in class $\Phi^2$. More precisely we will prove the following theorem.
\begin{thm} \label{main}
Suppose that $\sigma \in S^0_{1,0}(n,2)$ and $\phase_1,\phase_2 \in \Phi^2$ satisfy the strong non-dengeneracy condition. Then there exists a constant $C$ such that
\[
\|T_\sigma^{\phase_1,\phase_2}(f,g)\|_{L^1(\R^n)} \leq C\|f\|_{L^2(\R^n)}\|g\|_{L^2(\R^n)}
\]
for all Schwartz functions $f$ and $g$.
\end{thm}

The methods used to prove Theorem \ref{main} are significantly different from those employeed in \cite{GP} and \cite{RS}. As is frequently done, we apply different methods in different frequency regimes. The first is when either $\xi$ or $\eta$ is small. In this case we can write the bilinear operator as an iteration of linear operators, as in \cite{MRS3}, and use known results for linear operators. The second is when both $\xi$ and $\eta$ are large. Due to the symmetry of the operator and the fact we are interested in a bound in terms of the $L^2$-norm of both $f$ and $g$, we can further reduce this case to when $|\eta| \leq 2|\xi|$. Here we use a decomposition introduced in \cite{CM4} for the proof of the corresponding result for pseudodifferential operators. However, where they go on to use Carleson measure techniques, we must combine a quadratic $T(1)$-Theorem of M.~Christ and J.-L.~Journ\'e \cite{CJ} with commutator-type estimates.

To keep the notation as simple as possible, constants which can be easily estimated by given parameters are all denoted by $C$, even though the precise values will vary from from line to line. We also use the notation $A \lesssim B$, if there exists a constant $C$ such that $A\leq C B$. For clarity, we sometimes indicate the parameters on which a constant depends as subscripts.

\section{The Proof of Theorem \ref{main}}

We introduce a smooth function $\mu \colon \R^n \to \R$ such that $\mu(\xi) = 0$ for $|\xi| \leq 5$ and $\mu(\xi) = 1$ for $|\xi| \geq 6$. Observe that
\begin{equation}\label{eq:first_step}
\begin{aligned}
	T_{\sigma}^{\phase_1,\phase_2}(f,g) & = T_{\sigma}^{\phase_1,\phase_2}(\mu(D)f,\mu(D)g)+T_{\sigma}^{\phase_1,\phase_2}(\mu(D)f,(1-\mu)(D)g)\\
& \qquad +T_{\sigma}^{\phase_1,\phase_2}((1-\mu)(D) f,g),
\end{aligned}
\end{equation}
where $\mu(D)$ and $(1-\mu)(D)$ denote the Fourier multiplier operators given by $(\mu(D)f)\widehat{\,\,\,}(\xi)=\mu(\xi) \widehat{f}(\xi)$ and $((1-\mu)(D)f)\widehat{\,\,\,}(\xi)=(1-\mu)(\xi) \widehat{f}(\xi)$ respectively.

To estimate the $L^1$-norm of the last two terms in \eqref{eq:first_step}  we can make use of linear boundedness results by viewing the bilinear operator as an iteration of linear operators. We can write $T_{\sigma}^{\phase_1,\phase_2}((1-\mu)(D)f,g)$  as
\begin{align*}
T_{\sigma}^{\phase_1,\phase_2}((1-\mu)(D)f,g)(x) & = \int\left(\int \sigma(x,\xi,\eta)\widehat{g}(\eta)e^{i\phase_2(x,\eta)} \ddd\eta\right)(1-\mu)(\xi)\widehat{f}(\xi)e^{i\phase_1(x,\xi)} \ddd\xi\\
& = \int \mathfrak{a}_g(x,\xi) \widehat{f}(\xi)e^{i\phase_1(x,\xi)} \ddd\xi= T^{\phase_1}_{\mathfrak{a}_g}(f)(x),
\end{align*}
where
\[
\mathfrak{a}_g(x,\xi) = (1-\mu)(\xi)\int \sigma(x,\xi,\eta)\widehat{g}(\eta)e^{i\phase_2(x,\eta)} \ddd\eta.
\]
Using the fact that $\sigma \in S^0_{1,0}(n,2)\subset S^0_{0,0}(n,2)$ we get
\[
	\sup_{x,\xi,\eta\in \R^n}|\partial^\alpha_\xi\partial^\beta_\eta\partial^\gamma_x\sigma(x,\xi,\eta)| \lesssim 1.
\]
Therefore, bearing in mind the support properties of $\mathfrak{a}_g$, applying Theorem \ref{DW226} we find that
\[
	\sup_{\xi\in \R^n} \langle \xi\rangle^{-m+|\alpha|}\|\partial_\xi^\alpha \mathfrak{a}_g(\cdot,\xi)\|_{L^2(\R^n)} \lesssim \|g\|_{L^2(\R^n)}
\]
for any $m<0$, and so $\mathfrak{a}_g \in L^2S^m_1(n,1)$. Fixing $m < 0$ and applying Theorem \ref{SW} we see that $T^{\phase_1}_{\mathfrak{a}_g}$ is bounded from $L^2(\R^n)$ to $L^1(\R^n)$ with norm of size $\|g\|_{L^2(\R^n)}$. Consequently,
\[
\|T_{\sigma}^{\phase_1,\phase_2}((1-\mu)(D)f, g)\|_{L^1(\R^n)} \lesssim \|f\|_{L^2(\R^n)}\|g\|_{L^2(\R^n)}.
\]

Similarly, interchanging the roles of $f$ and $g$ in the previous argument, we can see that
\[
	\|T_{\sigma}^{\phase_1,\phase_2}(\mu(D)f,(1-\mu)(D)g)\|_{L^1(\R^n)} \lesssim \|f\|_{L^2(\R^n)}\|g\|_{L^2(\R^n)}.
\]

We now turn our attention to $T_{\sigma}^{\phase_1,\phase_2}(\mu(D)f,\mu(D)g)$. Let us introduce two smooth cut-off function $\chi,\nu \colon \R^{2n} \to \R$, such that $\chi(\xi,\eta) = 1$ for $|(\xi,\eta)| \leq 1$ and $\chi(\xi,\eta) = 0$ for $|(\xi,\eta)| \geq 2$, and $\nu(\xi,\eta) = 0$ for $2|\xi| \leq |\eta|$ and $\nu(\xi,\eta) = 1$ for $2|\eta| \leq |\xi|$.

Defining
\begin{align*}
\sigma_0(x,\xi,\eta) & =\chi(\xi,\eta)\sigma(x,\xi,\eta), \\
\sigma_1(x,\xi,\eta) & = (1-\chi(\xi,\eta))\nu(\xi,\eta)\sigma(x,\xi,\eta) \quad \text{and} \\
\sigma_2(x,\xi,\eta) & = (1-\chi(\xi,\eta))(1-\nu(\xi,\eta))\sigma(x,\xi,\eta),
\end{align*}
we have that $\sigma_0,\sigma_1,\sigma_2$ belong to the class $S^0_{1,0}(n,2)$ and we can decompose
\begin{equation} \label{low-high}
\begin{aligned}
T_{\sigma}^{\phase_1,\phase_2}(\mu(D)f,\mu(D)g)& = T_{\sigma_1}^{\phase_1,\phase_2}(\mu(D)f,\mu(D)g)+ T_{\sigma_2}^{\phase_1,\phase_2}(\mu(D)f,\mu(D)g).
\end{aligned}
\end{equation}
We observe that it suffices to control the $L^1$-norm of merely one of these terms, say $T_{\sigma_1}^{\phase_1,\phase_2}(\mu(D)f,\mu(D)g)$, because once again the other can be controlled in the same way by interchanging the roles of $f$ and $g$.

Following the analysis on pages 154--155 of \cite{CM4} we introduce an even real-valued smooth function $\psi$ whose Fourier transform is supported on the annulus $\{\xi \, | \, 1/2 \leq |\xi| \leq 2\}$ such that
\[
\int_0^\infty |\widehat{\psi}(t\xi)|^2 \frac{\dd t}{t} = 1
\]
for $\xi \neq 0$. Let $\theta$ be another real-valued smooth function whose Fourier transform is equal to one on the ball $\{\xi \, | \, |\xi| \leq 4\}$ and supported in $\{\xi \, | \, |\xi| \leq 5\}$. Then
\begin{equation} \label{rep1}
\begin{aligned}
& T_{\sigma_1}^{\phase_1,\phase_2}(\mu(D)f,\mu(D)g)(x)\\
& = \iiint_0^\infty\ \sigma_{1,t}(x,t\xi,t\eta)\widehat{\psi}(t\xi)\widehat{\theta}(t\eta)\mu(\xi)\widehat{f}(\xi)\mu(\eta)\widehat{g}(\eta)e^{i\phase_1(x,\xi)+i\phase_2(x,\eta)}  \frac{\dd t}{t}\ddd\xi\ddd \eta,
\end{aligned}
\end{equation}
for $\sigma_{1,t}(x,\xi/t,\eta/t) := \sigma_1(x,\xi/t,\eta/t)\widehat{\psi}(\xi)\widehat{\theta}(\eta)$. Using the Fourier inversion formula,
\begin{equation} \label{rep2}
\sigma_{1,t}(x,\xi,\eta) = \iint e^{i\xi\cdot u + i\eta\cdot v}m(t,x,u,v) \frac{\dd u \dd v}{(1 + |u|^2 + |v|^2)^N}
\end{equation}
where
\[
m(t,x,u,v) := \iint e^{-i\xi\cdot u - i\eta\cdot v} \big[(1-\Delta_\xi-\Delta_\eta)^N\sigma_{1,t}(x,\xi,\eta)\big] \ddd\xi\ddd \eta
\]
for any large fixed $N \in \N$.

Since the $(\xi,\eta)$-support of $\sigma_{1,t}$ is contained in a compact set independent of $t$ and $x$, and all $(x,\xi,\eta)$-derivatives are bounded independently of $t$, we see that $\partial^\alpha_xm(t,x,u,v)$ is bounded for each multi-index $\alpha$. Combining this with \eqref{rep1} and \eqref{rep2} we arrive at the representation
\begin{equation} \label{three}
\begin{aligned}
& T_{\sigma_1}^{\phase_1,\phase_2}(\mu(D)f,\mu(D)g)(x) \\
& = \iiint_0^\infty T^{\phase_1}_{\mu}(P^v_t(f))(x) T^{\phase_2}_{\mu}(Q^u_t(g))(x) \frac{m(t,x,u,v)}{(1+|u|^2+|v|^2)^N} \frac{\dd t}{t} \dd u \dd v
\end{aligned}
\end{equation}
for any large fixed $N \in \N$. Here
\begin{equation*}
T_{\mu}^{\phase_j}(f)(x) = \int \mu(\xi)\widehat{f}(\xi)e^{i\phase_j(x,\xi)} \ddd\xi\quad \text{for $j=1,2$,}
\end{equation*}
$P^v_t(f) = \theta_t^v * f$ and $Q^u_t(g) = \psi_t^u * g$, with $\theta_t^v(x) = t^{-n}\theta(x/t + v)$ and $\psi_t^u(x) = t^{-n}\psi(x/t + u)$.

We observe that
\[
T^{\phase_1}_{\mu}(P^v_t(f))(x) = \int \mu(\xi)\widehat{\theta}(t\xi)e^{it\xi\cdot v}\widehat{f}(\xi)e^{i\phase_1(x,\xi)}\ddd\xi.
\]
Since $\mu(\xi) = 0$ for $|\xi| \leq 5$ and $\widehat{\theta}(t\xi) = 0$ for $|t\xi| \geq 5$, then $\mu(\xi)\widehat{\theta}(t\xi) = 0$ for $t > 1$ and consequently $T^{\phase_1}_{\mu}(P^v_t(f))(x) = 0$ for $t > 1$. Using this fact, together with \eqref{three} and duality, to bound $\|T_{\sigma_1}^{\phase_1,\phase_2}(f,g)\|_{L^1(\R^n)}$ it suffices to control
\[
\left| \iint_0^1 T^{\phase_1}_{\mu}(P^v_t(f))(x) T^{\phase_2}_{\mu}(Q^u_t(g))(x) b(x) m(t,x,u,v) \frac{\dd t \dd x}{t}\right|
\]
with at most polynomial growth in $u$ and $v$ for arbitrary $b \in L^\infty(\R^n)$. Introducing the radial function $\psi_0$ whose Fourier transform is compactly supported on an annulus and equal to one on the support of the Fourier transform of $\psi$, we define $\psi_{0,t}(x) = t^{-n}\psi_0(x/t)$ and $Q_{0,t}(g) = \psi_{0,t} * g$. Let $M^{u,v}_t$ denote the $L^2$-bounded operator which is mulitiplication by $b(x) m(t,x,u,v)$, that is $M^{u,v}_t(f)(x) = b(x) m(t,x,u,v)f(x)$, and let $T^{\phase_2,*}_{\mu}$ denote the adjoint operator of $T^{\phase_2}_{\mu}$. Then $Q^u_t(g) = Q^u_t(Q_{0,t}(g))$ and so, using the Cauchy-Schwarz inequality,
\begin{align*}
& \left|\iint_0^1 T^{\phase_1}_{\mu}(P^v_t(f))(x) T^{\phase_2}_{\mu}(Q^u_t(g))(x) b(x)m(t,x,u,v) \frac{\dd t \dd x}{t}\right| \\
& = \left|\iint_0^1 (Q^u_tT^{\phase_2,*}_{\mu}M^{u,v}_t T^{\phase_1}_{\mu}P^v_t)(f)(x) Q_{0,t}(g)(x) \frac{\dd t \dd x}{t}\right| \\
& \leq \left(\iint_0^1 |(Q^u_tT^{\phase_2,*}_{\mu}M^{u,v}_t T^{\phase_1}_{\mu}P^v_t)(f)(x)|^2 \frac{\dd t \dd x}{t}\right)^{1/2} \left(\iint_0^\infty |Q_{0,t}(g)(x)|^2 \frac{\dd t \dd x}{t} \right)^{1/2} \\
& \lesssim \left(\iint_0^1 |(Q^u_tT^{\phase_2,*}_{\mu}M^{u,v}_t T^{\phase_1}_{\mu}P^v_t)(f)(x)|^2 \frac{\dd t \dd x}{t}\right)^{1/2} \|g\|_{L^2(\R^n)}.
\end{align*}
The last inequality follows by repeated application of Plancherel's Theorem and the fact that $\psi_0$ is supported on an annulus. Indeed,
\begin{equation} \label{plancherel}
\begin{aligned}
\left(\iint_0^\infty |Q_{0,t}(g)(x)|^2 \frac{\dd t \dd x}{t} \right)^{1/2} & = \left(\int_0^\infty\!\!\!\!\int  |\widehat{\psi}_0(t\xi)\widehat{g}(\xi)|^2 \frac{\ddd\xi\dd t}{t} \right)^{1/2} \\
& = \left(\iint_0^\infty |\widehat{\psi}_0(t\xi)|^2 \frac{\dd t}{t} |\widehat{g}(\xi)|^2 \ddd\xi\right)^{1/2} \\
& \lesssim \left(\int |\widehat{g}(\xi)|^2 \ddd\xi\right)^{1/2} = \left(\int |g(x)|^2 \dd x \right)^{1/2}.
\end{aligned}
\end{equation}
Therefore, the proof of Theorem \ref{main} will be complete if we can prove the quadratic estimate
\begin{equation} \label{aim}
\left(\iint_0^1 |(Q^u_tT^{\phase_2,*}_{\mu}M^{u,v}_t T^{\phase_1}_{\mu}P^v_t)(f)(x)|^2 \frac{\dd t \dd x}{t}\right)^{1/2} \lesssim \|b\|_{L^\infty(\R^n)}\|f\|_{L^2(\R^n)}.
\end{equation}

In what follows we will use the $L^2$-boundedness of the operators
$T^{\phase_1}_{\mu}$ and $T^{\phase_2}_{\mu}$, which is guaranteed by Theorem \ref{DW226} since $\mu\in S^0_{0,0}(n,1)$ and $\varphi_1,\varphi_2$ are strongly non-degenerate phases in $\Phi^2$.

To obtain \eqref{aim}, we wish to apply a quadratic $T(1)$-theorem. This requires two hypotheses: kernel estimates; and a cancellation condition (the $T(1)$ condition). Unfortunately it is not clear how to demonstrate either of these hypotheses for the operator $Q^u_tT^{\phase_2,*}_{\mu}M^{u,v}_t T^{\phase_1}_{\mu}P^v_t$ which appears in \eqref{aim}. However, let us suppose for a moment that we could commute operators at will. Define the operators $M_b$ and $M_{m}$ by $M_b(f)(x) = b(x)f(x)$ and $M_m(f)(x) = m(t,x,u,v)f(x)$, respectively. Because $M^{u,v}_t = M_mM_b$, commuting operators would lead us to consider $T^{\phase_2,*}_{\mu}M_mQ^u_tM_b P^v_tT^{\phase_1}_{\mu}$. This composition is much more amenable to the method mentioned above. This is because $T^{\phase_2,*}_{\mu}M_m$ is bounded on $L^2(\R^n)$, uniformly in $t$ and appears on the left of the composition, so can be disregarded when trying to prove a quadratic estimate. Furthermore $T^{\phase_1}_{\mu}$ is $t$-independent, $L^2$-bounded and appears on the right of the composition, so can also be disregarded. This leaves us with the task of proving a quadratic estimate for $Q^u_tM_b P^v_t$. This operator does satisfy kernel estimates and the $T(1)$ condition in this context is simply the fact that $Q^u_tM_b P^v_t(1) = Q^u_t(b)$ gives rise to a Carleson measure, which is well-known. For the remainder of the paper, we will fill in the details of this heuristic argument.

The following theorem gives us three equalities. The first makes precise the extent to which we may commute $T^{\phase_1}_{\mu}$ and $P^v_t$. Although we cannot commute $Q^u_t$ and $T^{\phase_2,*}_{\mu}$ in a manner which is acceptable to us, the second equality says there exist operators $U^u_t$ and $R_t$ such that $Q^u_tT^{\phase_2,*}_{\mu} = U_t^*R_t$ modulo an acceptable error, where $R_t$ has the same properties as $Q^u_t$ and $U^u_t$ is an operator bounded uniformly in $t$. The third equality shows us that the commutator of $R_t$ and $M_m$ is sufficiently well-behaved.
\begin{thm} \label{commutator}
For $0 < t \leq 1$, $u,v\in \R^n$, there exist operators $W^v_{1,t}$, $W^u_{2,t}$, $W^{u,v}_{3,t}$, $U^u_t$, $V^v_t$ and a radial smooth function $\rho$ supported in an annulus centred at the origin such that
\begin{enumerate}[label={\upshape (\roman*)}]

\item \label{part1}$[T^{\phase_1}_{\mu},P^v_t] := T^{\phase_1}_{\mu}P^v_t - P^v_tT^{\phase_1}_{\mu} = V^v_t + W^v_{1,t}$, and

\item \label{part2}$T^{\phase_2}_{\mu}Q^u_t = W^u_{2,t} + R_tU^u_t$,

\item \label{part3}$[R_t,M_m] := R_tM_m - M_mR_t = W^{u,v}_{3,t}$

\end{enumerate}
where $R_t$ is the multiplier operator defined by
\[
	R_t(f) = \int \rho(t\xi)\widehat{f}(\xi)e^{ix\cdot\xi}\ddd\xi.
\]
Moreover, there exists $\varepsilon>0$ such that for any $0 < t \leq 1$
\begin{align}
\|W^v_{1,t}(f)\|_{L^2(\R^n)} + \|W^u_{2,t}(f)\|_{L^2(\R^n)} + \|W^{u,v}_{3,t}(f)\|_{L^2(\R^n)} & \lesssim t^\varepsilon\|f\|_{L^2(\R^n)}, \notag \\
\|U^u_t(f)\|_{L^2(\R^n)} & \lesssim \|f\|_{L^2(\R^n)} \quad \text{and} \notag \\
\label{rq}\iint_0^1 |V^v_t(f)(x)|^2 \frac{\dd t}{t} \dd x & \lesssim \|f\|_{L^2(\R^n)}^2.
\end{align}
The implicit constants here depend on $\varepsilon$, but not on $t$ and only polynomially on $u$ and $v$.
\end{thm}
We will postpone the proof of this theorem until Section \ref{proofoflemma}. Using Theorem \ref{commutator}~\ref{part1}, we can compute
\begin{align*}
&\iint_0^1 |(Q^u_tT^{\phase_2,*}_{\mu}M^{u,v}_t T^{\phase_1}_{\mu}P^v_t)(f)(x)|^2 \frac{\dd t \dd x}{t}
 \lesssim\iint_0^1 |(Q^u_tT^{\phase_2,*}_{\mu}M^{u,v}_t P^v_tT^{\phase_1}_{\mu})(f)(x)|^2 \frac{\dd t \dd x}{t} \\
& \qquad +\iint_0^1 |(Q^u_tT^{\phase_2,*}_{\mu}M^{u,v}_t V^v_t)(f)(x)|^2 \frac{\dd t \dd x}{t} +\iint_0^1 |(Q^u_tT^{\phase_2,*}_{\mu}M^{u,v}_t W^v_{1,t})(f)(x)|^2 \frac{\dd t \dd x}{t} \\
& \lesssim\iint_0^1 |(Q^u_tT^{\phase_2,*}_{\mu}M^{u,v}_t P^v_tT^{\phase_1}_{\mu})(f)(x)|^2 \frac{\dd t \dd x}{t} + \|b\|_{L^\infty(\R^n)}^2\iint_0^1 |V^v_t(f)(x)|^2 \frac{\dd t \dd x}{t} \\
& \qquad + \|b\|_{L^\infty(\R^n)}^2\iint_0^1 |W^v_{1,t}(f)(x)|^2 \frac{\dd t \dd x}{t} \\
& \lesssim\iint_0^1 |(Q^u_tT^{\phase_2,*}_{\mu}M^{u,v}_t P^v_tT^{\phase_1}_{\mu})(f)(x)|^2 \frac{\dd t \dd x}{t} + \|b\|_{L^\infty(\R^n)}^2\|f\|_{L^2(\R^n)}^2.
\end{align*}
Using Theorem \ref{commutator}~\ref{part2} we have
\begin{align*}
&\iint_0^1 |(Q^u_tT^{\phase_2,*}_{\mu}M^{u,v}_t P^v_tT^{\phase_1}_{\mu})(f)(x)|^2 \frac{\dd t \dd x}{t} \\
& \lesssim\iint_0^1 |(U_t^{u,*}R_tM^{u,v}_t P^v_tT^{\phase_1}_{\mu})(f)(x)|^2 \frac{\dd t \dd x}{t} +\iint_0^1 |(W_{2,t}^{u,*}M^{u,v}_t P^v_tT^{\phase_1}_{\mu})(f)(x)|^2 \frac{\dd t \dd x}{t} \\
& \lesssim\iint_0^1 |(R_tM^{u,v}_t P^v_tT^{\phase_1}_{\mu})(f)(x)|^2 \frac{\dd t \dd x}{t} + \sup_{t>0}\int |(M^{u,v}_t P^v_tT^{\phase_1}_{\mu})(f)(x)|^2 \dd x \\
& \lesssim\iint_0^1 |(R_tM^{u,v}_t P^v_tT^{\phase_1}_{\mu})(f)(x)|^2 \frac{\dd t \dd x}{t} + \|b\|_{L^\infty(\R^n)}^2\|f\|_{L^2(\R^n)}.
\end{align*}
Finally, using Theorem \ref{commutator}~\ref{part3} we have
\begin{align*}
&\iint_0^1 |(R_tM^{u,v}_t P^v_tT^{\phase_1}_{\mu})(f)(x)|^2 \frac{\dd t \dd x}{t}=\iint_0^1 |(R_tM_mM_b P^v_tT^{\phase_1}_{\mu})(f)(x)|^2 \frac{\dd t \dd x}{t} \\
& \lesssim\iint_0^1 |(M_mR_tM_b P^v_tT^{\phase_1}_{\mu})(f)(x)|^2 \frac{\dd t \dd x}{t} +\iint_0^1 |(W^{u,v}_{3,t}M_b P^v_tT^{\phase_1}_{\mu})(f)(x)|^2 \frac{\dd t \dd x}{t} \\
& \lesssim\iint_0^1 |(M_mR_tM_b P^v_tT^{\phase_1}_{\mu})(f)(x)|^2 \frac{\dd t \dd x}{t} + \sup_{t>0}\int |(M_b P^v_tT^{\phase_1}_{\mu})(f)(x)|^2 \dd x \\
& \lesssim\iint_0^1 |(R_t M_b P^v_tT^{\phase_1}_{\mu})(f)(x)|^2 \frac{\dd t \dd x}{t} + \|b\|_{L^\infty(\R^n)}^2\|f\|_{L^2(\R^n)}.
\end{align*}
Given that $T^{\phase_1}_{\mu}$ is an $L^2$-bounded operator and independent of $t$, to prove \eqref{aim} we only need to prove
\begin{equation} \label{aim3}
\iint_0^\infty |(R_tM_b P^v_t)(f)(x)|^2 \frac{\dd t \dd x}{t} \lesssim \|b\|_{L^\infty(\R^n)}^2\|f\|_{L^2(\R^n)}^2.
\end{equation}

Now we recall Theorem 1 from \cite{CJ}, which is the quadratic $T(1)$-theorem that will be useful to us.\footnote{This result is implicit in the work of R.~Coifman and Y.~Meyer \cite{CM5}. Now there are also generalisations in the form of local $T(b)$-theorems, see, for example, S.~Hofmann \cite{Ho}.} It concerns a family of operators $\{T_t\}_{t>0}$ which are defined as integration against a kernel:
\begin{equation} \label{oper}
T_t(f)(x) = \int K_t(x,y)f(y) \dd y.
\end{equation}
Estimates of significance for the kernel $K_t$ are
\begin{align}
|K_t(x,y)| & \leq \frac{C}{t^n} \frac{1}{\left(1 + |x-y|/t\right)^{n+1}} \quad \text{and} \label{ker1} \\
|K_t(x,y+h) - K_t(x,y)| & \leq \frac{C}{t^n} \frac{|h/t|}{\left(1 + |x-y|/t\right)^{n+1}} \label{ker2} %\\
%|K_t(x+h,y) - K_t(x,y)| & \leq \frac{C}{t^n} \frac{|h/t|}{\left(1 + |x-y|/t\right)^{n+1}}, \label{ker3}
\end{align}
for $|h| \leq t/2$.
\begin{thm} \label{tbglobal}
Suppose the kernel of $T_t$ given by \eqref{oper} satisfies estimates \eqref{ker1} and \eqref{ker2} for some constant $C < \infty$ together with the Carleson measure condition
\begin{equation} \label{carleson}
\sup_Q \frac{1}{|Q|} \int_0^{\ell(Q)}\!\! \int_Q |T_t(1)(x)|^2 \frac{\dd t \dd x}{t} \leq C.
\end{equation}
Then the quadratic estimate
\begin{equation*} %\label{qe}
\iint_0^\infty |T_t(f)(x)|^2 \frac{\dd t \dd x}{t} \lesssim \|f\|_{L^2(\R^n)}^2.
\end{equation*}
holds.
\end{thm}

Now, we have that
\[
|x|^N |\partial^\alpha\theta(x+v)| \lesssim |x + v|^N |\partial^\alpha\theta(x+v)| + |v|^N |\partial^\alpha\theta(x+v)|
\lesssim 1 + |v|^N
\]
for any $N \geq 0$ and multi-index $\alpha$, since $\theta$ is a Schwartz function. Therefore
\[
|\partial^\alpha\theta^v_t(x)| \lesssim \frac{\langle v\rangle^N}{t^{n+|\alpha|}\langle x/t\rangle^N}
\]
and so $\theta^v_t(x-y)$ satisfies \eqref{ker1} with a constant $C$ that only depends polynomially on $v$. By the mean-value theorem, $\theta^v_t(x-y)$ also satisfies \eqref{ker2}. It is even easier to see that the kernel of $R_t$, which we can write in the form $m_t(x-y) := \rho(t\cdot)\inverse{\,}(x-y)$, also satisfies \eqref{ker1}. From these estimates it follows that the kernel of $(R_tM_b P^v_t)$
\begin{equation} \label{kerdef}
K_t(x,y) = \int m_t(x-z)b(z)\theta^v_t(z-y)\dd z
\end{equation}
also satisfies estimates \eqref{ker1} and \eqref{ker2}. This may be seen by splitting the integral in \eqref{kerdef} into integration over two half-planes $H_x$ and $H_y$ containing $x$ and $y$, respectively, with common boundary being the hyperplane perpendicular to the line segment passing though $x$ and $y$ and containing the mid-point $(x+y)/2$. Clearly then $|y-z| \geq |x-y|/2$ when $z \in H_x$ and $|x-z| \geq |x-y|/2$ when $z \in H_y$ so we may conclude
\begin{align*}
& \left|\int m_t(x-z)b(z)\theta^v_t(z-y)\dd z\right| \\
& = \left|\int_{H_x} m_t(x-z)b(z)\theta^v_t(z-y)\dd z + \int_{H_y} m_t(x-z)b(z)\theta^v_t(z-y)\dd z\right| \\
& \lesssim \int_{H_x} \langle(x-z)/t\rangle^{-(n+1)} \langle(z-y)/t\rangle^{-(n+1)} t^{-2n} \dd z \\
& \qquad + \int_{H_y} \langle(x-z)/t\rangle^{-(n+1)} \langle(z-y)/t\rangle^{-(n+1)} t^{-2n} \dd z \\
& \lesssim  \langle(x-y)/t\rangle^{-(n+1)} \int_{H_x} \langle(x-z)/t\rangle^{-(n+1)} t^{-n} \dd z \\
& \qquad + \langle(x-y)/t\rangle^{-(n+1)} \int_{H_y} \langle(z-y)/t\rangle^{-(n+1)} t^{-2n} \dd z \\
& \lesssim  t^{-n}\langle(x-y)/t\rangle^{-(n+1)}.
\end{align*}
This proves \eqref{ker1} and \eqref{ker2} follows similarly (see \cite{GT2} and also page A-36 in \cite{G}).

The measure
\[
|(R_tM_b P^v_t)(1)(x)|^2 \frac{\dd t \dd x}{t} = |R_t(b)(x)|^2 \frac{\dd t \dd x}{t}
\]
is a Carleson measure, since $b \in L^\infty(\R^n) \subset \text{BMO}(\R^n)$ and $m_t$ has mean-value zero (see Theorem 3 on page 159 of \cite{S}), whence we have estimate \eqref{carleson}. Therefore, we may apply Theorem \ref{tbglobal} and obtain \eqref{aim3} as required.

\section{Proof of Theorem \ref{commutator}} \label{proofoflemma}

Both $Q^u_t$ and $P^v_t$ are convolutions with smooth functions and hence multipliers. To deal with both such operators together we define
\[
R_t(f) = \int \rho(t\xi)\widehat{f}(\xi)e^{ix\cdot\xi}\ddd\xi,
\]
where we have in mind that $\rho$ will eventually be assumed to have properties similar to those of $\widehat{\psi}$ or $\widehat{\theta}$. To be precise, it is sufficient to assume $(x,\xi) \mapsto \rho(\xi)$ belongs to $S^0_{1,0}(n,1)$. Define
\begin{equation} \label{four}
T_{a_t}^{\phase}(f)(x) = \int a_t(\xi)\widehat{f}(\xi)e^{i\phase(x,\xi)} d\xi,
\end{equation}
for a smooth amplitude $(x,\xi) \mapsto a_t(\xi)$ which belongs to $S^0_{0,0}(n,1)$ uniformly in $t \in (0,1)$ and $\phase \in \Phi^2$ satisfying the strong non-degeneracy condition. Now we can compute
\begin{align*}
(R_{t}T_{a_t}^{\phase})(f)(x) & = \int \rho(t\eta)\left(\int \left( \int a_t(\xi)\widehat{f}(\xi)e^{i\phase(y,\xi)} \ddd\xi\right)e^{-iy\cdot\eta} \dd y \right)e^{ix\cdot\eta} \ddd\eta \\
& = \iiint \rho(t\eta)a_t(\xi) e^{i(x-y)\cdot\eta + i\phase(y,\xi)-i\phase(x,\xi)}\widehat{f}(\xi)e^{i\phase(x,\xi)} \ddd\eta \ddd\xi\dd y \\
& := \int \sigma_{t}(x,\xi)\widehat{f}(\xi)e^{i\phase(x,\xi)}\ddd\xi,
\end{align*}
where
\begin{equation} \label{product}
\sigma_{t}(x,\xi):= a_t(\xi)\iint \rho(t\eta) e^{i(x-y)\cdot\eta + i\phase(y,\xi)-i\phase(x,\xi)} \ddd\eta \dd y.
\end{equation}
Therefore $R_{t}T_{\sigma}^{\phase}$ can be represented as a Fourier integral operator with phase function $\phase$ and amplitude $\sigma_{t}$. Now we would like to understand the behaviour of $\sigma_{t}$. To this end we start with the following proposition.

%%%%%%%%%%%%%%
\begin{prop}\label{left composition with pseudo}
Assume that $\phase \in \mathcal{C}^{\infty}(\R^n \times \R^n)$ is such that
\begin{enumerate}[label={\upshape (\roman*)}]

\item \label{partone}for constants $C_1$ and $C_2$ one has $C_1 |\xi|\leq |\nabla_{x} \phase(x,\xi)| \leq C_2 |\xi|$ for all $x,\xi \in \R^n$, and

\item \label{parttwo}for all $|\alpha|,|\beta|\geq 1$ one has $|\partial^{\alpha}_{x} \phase(x,\xi)|\leq C_{\alpha}\langle \xi \rangle$ and $|\partial^{\alpha}_{\xi}\partial^{\beta}_{x} \phase(x,\xi)|\leq C_{\alpha,\beta},$ for all $x,\xi\in \R^n$.

\end{enumerate}
Then, for each $\varepsilon \in (\frac{1}{2},1)$, there exist $M=M(\varepsilon)$ and $\mu=\mu(\varepsilon)$, both greater than zero, such that \eqref{product} can be written as
\begin{equation}  \label{main left composition estim}
\sigma_{t}(x,\xi) = \rho(t\nabla_{x}\phase(x,\xi))a_t(\xi) + \sum_{0<|\alpha| < M}\frac{t^{|\alpha|}}{\alpha!}\, \sigma_{\alpha}(t,x,\xi)+t^{\mu} r(t,x,\xi),
\end{equation}
for $t\in [0,1]$, where for all multi-indices $\beta$ and $\gamma$ one has
\begin{align*}
|\partial^{\gamma}_{\xi} \partial^{\beta}_{x}\sigma_{\alpha}(t,x,\xi)| & \leq C_{\alpha, \beta, \gamma}\,t^{-\varepsilon|\alpha|} \quad \text{and} \\
|\partial^{\alpha}_{\xi} \partial^{\beta}_{x} r(t,x,\xi)| & \leq C_{\alpha, \beta}.
\end{align*}
\end{prop}
\begin{proof}
First we remark that it is sufficient to prove the proposition in the special case $a_t(\xi) \equiv 1$ in \eqref{product}. The general case may be obtained by multiplying \eqref{main left composition estim} (with $a_t(\xi)$ replaced by the constant $1$) by $a_t(\xi)$ and observing that all the claimed properties hold for the product since $a_t \in S^0_{0,0}(n,1)$.

Let $\chi(x-y)\in \mathcal{C}^{\infty}(\mathbb{R}^n \times \R^n)$ such that $0\leq \chi\leq 1,$ $\chi(x-y)=1$ for $|x-y|<\frac{\varepsilon}{2}$ and $\chi(x-y)=0$ for $|x-y|>\varepsilon$. We now decompose $\sigma_t(x,\xi)$ into two parts $\textbf{I}_1 (t,x,\xi)$ and $\textbf{I}_2 (t,x,\xi)$ where
\begin{equation*}
  \textbf{I}_1 (t,x,\xi)=\iint\rho(t\eta)\, (1-\chi(x-y))\,e^{i(x-y)\cdot\eta+i\phase(y,\xi)-i\phase(x,\xi)}\,d\eta dy,
  \end{equation*}
  and
\begin{equation*}
  \textbf{I}_2 (t,x,\xi)=\iint\rho(t\eta)\,\chi(x-y)\,e^{i(x-y)\cdot\eta+i\phase(y,\xi)-i\phase(x,\xi)}\,d\eta dy.
  \end{equation*}

We begin by analysing $\textbf{I}_1 (t,x,\xi)$. To this end we introduce the differential operators
\[
^{t}L_{\eta} =-i\sum_{j=1}^{n}\frac{x_j -y_j}{|x-y|^2}\partial_{\eta_j} \quad \text{and} \quad
^{t}L_{y} =\frac{1}{|\nabla_{y}\phase(y,\xi)|^2 -i\Delta_{y}\phase(y,\xi)}(1-\Delta_{y}).
\]
Because of the conditions on the phase function one has $|\langle \nabla_{y}\phase(y,\xi) \rangle^2 -i\Delta_{y}\phase(y,\xi)|\geq |\nabla_{y}\phase(y,\xi)|^2\geq C_1 \xxi^2$. Now integration by parts yields
\begin{equation*}
  \textbf{I}_1 (t,x,\xi)=\iint L^{N_{2}}_{y} \{e^{-iy\cdot\eta}L^{N_1}_{\eta}[(1-\chi(x-y))\rho(t\eta)]\}\,e^{ix\cdot\eta+i\phase(y,\xi)-i\phase(x,\xi)}\,\ddd\eta\, dy.
  \end{equation*}
Now since $t\leq 1$, provided $M<N_1$, we have
\begin{align*}
|\partial^{N_1}_{\eta_j}\rho(t\eta)| & \leq C_{N_{1}}\, t^{N_1} \langle t\eta\rangle^{-N_1}= C_{N_{1}}\, t^{N_1} \langle t\eta\rangle^{-M}\langle t\eta\rangle^{-(N_1 -M)} \\
& \leq C_{N_{1}}\, t^{N_{1}} (t^2+|t\eta|^2)^{-M/2} \langle t\eta\rangle^{-(N_{1}-M)}\leq C_{N_1}\, t^{N_{1}-M}\langle \eta\rangle^{-M}.
\end{align*}
Therefore, choosing $N_2 < M < N_1$ large enough we have,
\begin{equation*}
  | \textbf{I}_1 (t,x,\xi)| \leq t^{N_1 -M}\xxi^{-2N_2}\iint_{|x-y|>\varepsilon} \langle \eta\rangle^{2N_2} |x-y|^{-2N_1} \langle\eta\rangle^{-M} \,d\eta\, dy\lesssim  t^{N_1 -M} \xxi^{-2N_2}.
  \end{equation*}
Estimating derivatives of $\textbf{I}_1(t,x,\xi)$ with respect to $x$ and $\xi$ may introduce factors estimated by powers of $\xxi$, $\langle \eta\rangle$ and $|x-y|$, which can all be handled by choosing $N_1$ and $N_2$ appropriately. Therefore for all $N$ and some $\nu>0$
\begin{equation*}
  | \partial^{\alpha}_{\xi} \partial^{\beta}_{x}\textbf{I}_1 (t,x,\xi)| \leq C_{\alpha, \beta}\, t^{\nu}\,\xxi^{-N}
  \end{equation*}
and so $t^{-\nu}\textbf{I}_1 (t,x,\xi)$ forms part of the error term $r(t,x,\xi)$ in \eqref{main left composition estim}.

We now proceed to the analysis of $\textbf{I}_2 (t,x,\xi)$. First we make the change of variables $\eta=\nabla_x\phase(x,\xi)+\zeta$ in the integral defining $\textbf{I}_2 (t,x,\xi)$ and then expand $\rho(t\eta)$ in a Taylor series to obtain
\begin{equation*}\label{Eq:ralpha}
  \begin{aligned} \rho(t\nabla_x\phase(x,\xi)+t\zeta) & = \sum_{0 \leq |\alpha|<M} t^{|\alpha|}\frac{\zeta^\alpha}{\alpha !} (\partial_\xi^\alpha\rho)(t\nabla_x\phase(x,\xi)) + t^{M} \sum_{|\alpha|=M} C_\alpha {\zeta^\alpha} r_\alpha(x,\xi,\zeta), \\ \text{where} \quad r_\alpha(x,\xi,\zeta) & = \int_0^1 (1-s)^{M-1} (\partial_\xi^{\alpha} \rho)(t\nabla_x\phase(x,\xi)+st\zeta) ds.
  \end{aligned} \end{equation*}
If we set
\[
\Phi(x,y,\xi)=\phase(y,\xi)-\phase(x,\xi)+(x-y)\cdot\nabla_x\phase(x,\xi),
\]
we obtain
\[
\textbf{I}_2 (t,x,\xi)= \sum_{|\alpha|<M} \frac{t^{|\alpha|}}{\alpha!}\, \sigma_{\alpha}(t,x,\xi) + t^{M}\, \sum_{|\alpha|=M} C_\alpha\, R_{\alpha}(t,x,\xi),
\]
where
\begin{align*}
\sigma_{\alpha}(t,x,\xi) & = \iint e^{i(x-y)\cdot\zeta+i\Phi(x,y,\xi)} \zeta^{\alpha}\,\chi(x-y)\, (\partial_\xi^\alpha \rho)(t\nabla_x\phase(x,\xi))\, dy \ddd\zeta \\
& = (\partial_\xi^\alpha \rho)(t\nabla_x\phase(x,\xi)) \partial_y^{\alpha}\left.\left[ e^{i\Phi(x,y,\xi)}\chi(x-y) \right]\right|_{y=x}
\end{align*}
and
\[
R_{\alpha}(t,x,\xi) = \iint e^{i(x-y)\cdot\zeta} e^{i\Phi(x,y,\xi)} \zeta^{\alpha} \chi(x-y) \, r_\alpha(t, x,\xi,\zeta) dy \ddd\zeta.
\]

We now claim that
\begin{equation} \label{half}
\left|\left.\partial_y^{\gamma} e^{i\Phi(x,y,\xi)}\right|_{y=x}\right|\lesssim \bra{\xi}^{|\gamma|/2}.
\end{equation}
We first observe that when $\gamma = 0$, \eqref{half} is obvious. To obtain \eqref{half} for $\gamma \neq 0$ we recall Faa-Di Bruno's formula:
\[
\partial_y^{\gamma} e^{i\Phi(x,y,\xi)}=\sum_{\gamma_1 + \cdots+ \gamma_k =\gamma} C_\gamma (\partial^{\gamma_{1}}_{y}\Phi(x,y,\xi))\cdots (\partial^{\gamma_{k}}_{y}\Phi(x,y,\xi))\,e^{i\Phi(x,y,\xi)}.
\]
where the sum ranges of $\gamma_j$ such that $|\gamma_{j}|\geq 1$ for $j=1,2,\dots, k$ and $\gamma_1 + \cdots+ \gamma_k =\gamma$ for some $k \in \N$. Since $\Phi(x,x,\xi)=0$ and $\left.\partial_y\Phi(x,y,\xi)\right|_{y=x}=0$, setting $y=x$ in the expansion above leaves only terms in which $|\gamma_j|\geq 2$ for all $j = 1,2,\dots,k$. But $\sum_{j=1}^{k} |\gamma_j |\leq |\gamma|,$ so we actually have $2k\leq |\gamma|$, that is $k\leq \frac{|\gamma|}{2}$. Assumption \ref{parttwo} on the phase tells us that $|\partial^{\gamma_{j}}_{y}\Phi(x,y,\xi)| \lesssim \xxi$, so
\[
\left|\left.\partial_y^{\gamma} e^{i\Phi(x,y,\xi)}\right|_{y=x}\right| \lesssim \xxi \cdots \xxi \lesssim \xxi ^{k} \lesssim \bra{\xi}^{|\gamma|/2},
\]
which is \eqref{half}.

If we use the fact that $t\leq 1$ and assumption \ref{partone} on the phase function $\phase$, we obtain
\[
|\sigma_{\alpha}(t,x,\xi)| \lesssim \bra{t\nabla_x\phase(x,\xi)}^{-|\alpha|} \bra{\xi}^{\frac{|\alpha|}{2}}
\lesssim t^{-\frac{|\alpha|}{2}}\bra{t\nabla_x\phase(x,\xi)}^{-|\alpha|} \bra{t\xi}^{\frac{|\alpha|}{2}}
\lesssim t^{-\frac{|\alpha|}{2}},
\]
when $|\alpha| > 0$ and, clearly, $\sigma_{\alpha}(t,x,\xi) = \rho(t\nabla_x\phase(x,\xi))$ when $\alpha = 0$.

The derivatives of $\sigma_{\alpha}$ with respect to $x$ or $\xi$ do not change estimates when applied to $\rho$ by the assumptions of the lemma. When applied to $\partial_y^{\alpha} e^{i\Phi(x,y,\xi)}|_{y=x}$ they do not change estimates since $|\partial_\xi^\beta\partial_x^\alpha\phase(x,\xi)|\leq C_{\alpha\beta}.$ Therefore for all multi-indices $\beta$, $\gamma\in \mathbb{Z}_{+}$,
\[
|\partial^{\beta}_{\xi} \partial^{\gamma}_{x} \sigma_{\alpha}(t,x,\xi)| \leq C_{\beta, \gamma}  t^{-\frac{|\alpha|}{2}}.
\]
as required.

To estimate the remainder $R_{\alpha}$, we take $g\in \mathcal{C}_0^\infty(\R^n)$ such that $g(x)=1$ for $|x|<r/2$ and $g(x)=0$ for $|x|>r$, for some small $r>0$ to be chosen later. We then decompose
\begin{equation*}\label{EQ:Ralphas}
 \begin{aligned} R_{\alpha}(t,x,\xi) &= R_\alpha^I(t,x,\xi)+ R_\alpha^{I\!\!I}(t,x,\xi) \\ & = \iint e^{i(x-y)\cdot\zeta} g\left(\frac{\zeta}{\bra{\xi}}\right) D_y^{\alpha} \left[ e^{i\Phi(x,y,\xi)}\, \chi(x-y)\,r_\alpha(t,x,\xi,\zeta) \right] \dd y \ddd\zeta \\ & \qquad + \iint e^{i(x-y)\cdot\zeta} \left(1-g\left( \frac{\zeta}{\bra{\xi}}\right)\right) D_y^{\alpha}\left[  e^{i\Phi(x,y,\xi)}\,\chi(x-y)\, r_\alpha(t,x,\xi,\zeta)\right] \dd y \ddd\zeta.
   \end{aligned}
   \end{equation*}

As a preamble to estimating $R_\alpha^I(t,x,\xi)$, we note that the inequality
\[
\bra{\xi}\leq 1+|\xi|\leq \sqrt{2}\bra{\xi},
\]
and assumption \ref{partone} on the phase function $\phase$ yield
\[
\begin{aligned}
\bra{\nabla_x\phase(x,\xi)+s\zeta}\leq & (C_2\sqrt{2}+r)\bra{\xi} \\ \sqrt{2}\bra{\nabla_x\phase(x,\xi)+s\zeta}\geq & 1+|\nabla_x\phase|-|\zeta| \geq \bra{\nabla_x\phase}-|\zeta|\geq (C_1-r)\bra{\xi}.
\end{aligned}
\]
Therefore, if we choose $r<C_1,$ then for any $s\in (0,1)$, $\bra{\nabla_x\phase(x,\xi)+s\zeta}$ and $\bra{\xi}$ are equivalent.

This yields that for $|\zeta|\leq r\bra{\xi}$, $r_\alpha(t, x,\xi,\zeta)$ and all of its derivatives are dominated by $ \bra{t\xi}^{-|\alpha|}.$ Furthermore, for $t\leq 1$, it follows from the properties of $r_\alpha$ that
\begin{equation}\label{eq:estr}
\begin{aligned}
\left|\partial_\zeta^{\beta}\left( g\left( \frac{\zeta}{\bra{\xi}}\right) r_\alpha(t,x,\xi,\zeta)\right)\right| & \leq C_{\alpha,\beta}\sum_{\gamma\leq\beta} \left|\partial_\zeta^\gamma g\left(\frac{\zeta}{\bra{\xi}}\right) \partial_\zeta^{\beta-\gamma} r_\alpha(t,x,\xi,\zeta)\right| \\
& \leq C_{\alpha,\beta} \sum_{\gamma\leq\beta}t^{|\beta|-|\gamma|}\bra{\xi}^{-|\gamma|} \bra{t\xi}^{-|\alpha|-|\beta|+|\gamma|} \\ & \leq C_{\alpha,\beta} t^{-\varepsilon|\alpha|} \,\bra{\xi}^{-|\beta|-\varepsilon|\alpha|},
\end{aligned}
\end{equation}
for all $\varepsilon\in (\frac{1}{2}, 1)$.

At this point we also need estimates for $\partial_y^\alpha e^{i\Phi(x,y,\xi)}$ off the diagonal, that is, when $x\neq y.$ This derivative has at most $|\alpha|$ powers of terms $\nabla_y\phase(y,\xi)-\nabla_x\phase(x,\xi)$, possibly also multiplied by at most $|\alpha|$ higher order derivatives $\partial_y^\beta\phase(y,\xi)$, which can be estimated by $(|y-x|\bra{\xi})^{|\alpha|}$ by property \ref{parttwo} of the phase function. The term containing no difference $\nabla_y\phase(y,\xi)-\nabla_x\phase(x,\xi)$ is the product of at most $|\alpha|/2$ terms of the type $\partial_y^\beta\phase(y,\xi)$, which can be estimated by $\bra{\xi}^{|\alpha|/2}$ in view of property \ref{parttwo}. These observations yield
\[
|\partial_y^\alpha e^{i\Phi(x,y,\xi)}|\leq C_\alpha (1+|x-y|\bra{\xi})^{|\alpha|} \bra{\xi}^{|\alpha|/2},
\]
and therefore we also have
\begin{equation}\label{eq:ests}
\left|\partial_y^{\alpha}\left[ e^{i\Phi(x,y,\xi)} \chi(x-y) \right] \right|\lesssim (1+|x-y|\bra{\xi})^{|\alpha|}\bra{\xi}^{\frac{|\alpha|}{2}}.
\end{equation}

Now to estimate $R^I_{\alpha}(t,x,\xi),$ let
\[
L_\zeta=\frac{(1-\bra{\xi}^2\Delta_\zeta)}{1+\bra{\xi}^2|x-y|^2}, \quad \text{so} \quad L_\zeta^N e^{i(x-y)\cdot\zeta}=e^{i(x-y)\cdot\zeta}.
\]
Integration by parts with $L_\zeta$ yield
\[
\begin{aligned}
& R^I_{\alpha}(t,x,\xi) \\
& = \iint \frac{e^{i(x-y)\cdot\zeta}\,\partial_y^{\alpha}\left[ \chi(x-y)\, e^{i\Phi(x,y,\xi)}\right]}{(1+\bra{\xi}^2 |x-y|^2)^N} (1-\bra{\xi}^{2}\Delta_\zeta)^N \left\{ g\left(\frac{\zeta}{\bra{\xi}}\right) \, r_\alpha(t,x,\xi,\zeta) \right\} \dd y \ddd\zeta \\
& = \iint \frac{e^{i(x-y)\cdot\zeta}\,\partial_y^{\alpha}\left[\chi(x-y)\,e^{i\Phi(x,y,\xi)}\right]}{(1+\bra{\xi}^2 |x-y|^2)^N} \sum_{|\beta|\leq 2N} c_{\beta}\bra{\xi}^{|\beta|} \left\{ \partial_\zeta^{\beta}\left( g\left(\frac{\zeta}{\bra{\xi}}\right) r_\alpha(t,x,\xi,\zeta)\right) \right\} \dd y \ddd\zeta.
\end{aligned}
\]
Using estimates (\ref{eq:estr}), (\ref{eq:ests}) and that the size of the support of $g(\zeta/\bra{\xi})$ in $\zeta$ is bounded by $(r\bra{\xi})^n$, yield
\[
\begin{aligned}
|R^I_{\alpha}(t,x,\xi)| & \leq C\,t^{-\varepsilon|\alpha|}\,\sum_{|\beta|\leq 2N} \bra{\xi}^{n+|\beta|} \bra{\xi}^{-\varepsilon|\alpha|-|\beta|} \bra{\xi}^{\frac{|\alpha|}{2}} \int_{|x-y|<\varepsilon}\frac{(1+|x-y|\bra{\xi})^{|\alpha|}} {(1+\bra{\xi}^2 |x-y|^2)^N} \dd y \\ & \leq C\, t^{-\varepsilon|\alpha|} \bra{\xi}^{2n+(\frac{1}{2}-\varepsilon)|\alpha|},
\end{aligned}
\]
if we choose $2N>n$, and the constant $C$ is independent of $t$ (because of (\ref{eq:estr})). The derivatives of $R_{\alpha}^I(t,x,\xi)$ with respect to $x$ and $\xi$ give an extra power of $\zeta$ under the integral. This amounts to taking more $y$-derivatives, yielding a higher power of $\bra{\xi}.$ However, for a given number of derivatives of the remainder $R_{\alpha}^I(t,x,\xi)$, we are free to choose $M=|\alpha|$ as large as we like and therefore the higher power of $\bra{\xi}$ will not cause a problem. Thus for all multi-indices $\beta$, $\gamma\in \mathbb{Z}_{+},$ all $\varepsilon\in (\frac{1}{2},1)$ and all $|\alpha|>\frac{4n}{2\varepsilon-1},$ we have
\begin{equation*}
|\partial^{\beta}_{\xi} \partial^{\gamma}_{x} R_{\alpha}^{I}(t,x,\xi)| \leq C_{\beta, \gamma} t^{-\varepsilon|\alpha|},
\end{equation*}
where the constant $C_{\beta, \gamma}$ does not depend on $t$.

Finally to estimate $R_{\alpha}^{I\!\!I}(t,x,\xi)$ one defines
\[
\Psi(x,y,\xi,\zeta)=(x-y)\cdot\zeta+\Phi(x,y,\xi)= (x-y)\cdot(\nabla_x\phase(x,\xi)+\zeta)+\phase(y,\xi)-\phase(x,\xi).
\]
It follows from assumptions \ref{partone} and \ref{parttwo} on the phase function $\phase$ that if we choose $\varepsilon <r/2C_0,$ then since $|x-y|<\varepsilon$ in the support of $\chi$, one has (using that we are in the region $|\zeta|\geq r\bra{\xi}$)
\begin{equation*}\label{eq:rho}
\begin{aligned}
|\nabla_y\Psi| & =|-\zeta+\nabla_y\phase-\nabla_x\phase|\leq 2C_2(|\zeta|+\bra{\xi}), \quad \text{and} \\
|\nabla_y\Psi| & \geq |\zeta|-|\nabla_y\phase-\nabla_x\phase| \geq \frac{1}{2}|\zeta|+\left(\frac{r}{2}-C_0|x-y| \right)\bra{\xi}\geq C(|\zeta|+\bra{\xi}).
\end{aligned}
\end{equation*}
Now, since $|\partial_y^\beta\phase(y,\xi)|\leq C_\beta\bra{\xi}$, for any $\beta$ we have the estimate
\begin{equation}\label{no idea what to call 1}
|\partial_y^\beta(e^{-i\Phi(x,y,\xi)}\d_y^{\gamma} e^{i\Phi(x,y,\xi)})(x,y,\xi)|\lesssim\bra{\xi}^{|\gamma|}.
\end{equation}
For $M=|\alpha|>0$ we also observe that
\begin{equation} \label{eq:rs}
|r_\alpha(t, x,\xi,\zeta)|\leq C_\alpha.
\end{equation}
For the differential operator defined to be $^tL_y=i|\nabla_y\Psi|^{-2}\sum_{j=1}^n (\partial_{y_j}\Psi) \partial_{y_j}$, induction shows that $L_y^N$ has the form
\begin{equation*}\label{EQ:LNy} L_y^N=\frac{1}{|\nabla_y\Psi|^{4N}}\sum_{|\beta|\leq N} P_{\beta,N}\partial_y^\beta,\quad \text{where} \quad P_{\beta,N}=\sum_{|\mu|=2N} c_{\beta\mu\delta_j}(\nabla_y\Psi)^\mu \partial_y^{\delta_1}\Psi\cdots \partial_y^{\delta_N}\Psi,
\end{equation*}
$|\mu|=2N$, $|\delta_j|\geq 1$ and $\sum_{M}^N |\delta_j|+|\beta|=2N$. It follows from assumption \ref{parttwo} on $\phase$ that $|P_{\beta,N}|\leq C(|\zeta|+\bra{\xi})^{3N}.$ Now Leibniz's rule yields
\begin{equation*}\label{EQ:RII}
\begin{aligned}
& R^{I\!\!I}_{\alpha}(t,x,\xi) = \iint e^{i(x-y)\cdot\zeta} \left(1-g\left( \frac{\zeta}{\bra{\xi}}\right)\right) r_\alpha(x,\xi,\zeta) D_y^{\alpha}\left[ e^{i\Phi(x,y,\xi)}\chi(x-y) \right] \dd y \ddd\zeta \\
& = \iint e^{i\Psi(x,y,\xi,\zeta)} \left(1-g\left(\frac{\zeta}{\bra{\xi}}\right)\right) r_\alpha(t,x,\xi,\zeta) \\
& \qquad \times \sum_{\gamma_1+\gamma_2=\alpha} (e^{-i\Phi(x,y,\xi)}D_y^{\gamma_1} e^{i\Phi(x,y,\xi)})\, D^{\gamma_2}_{y}\chi(x-y) \dd y \ddd\zeta \\
& = \iint e^{i\Psi(x,y,\xi,\zeta)} |\nabla_y\Psi|^{-4N}\sum_{|\beta|\leq N} P_{\beta,N}(x,y,\xi,\zeta) \left(1-g\left(\frac{\zeta}{\bra{\xi}} \right)\right) r_\alpha(t,x,\xi,\zeta) \\
& \qquad \times \sum_{\gamma_1+\gamma_2=\alpha} \partial_y^\beta [(e^{-i\Phi(x,y,\xi)}D_y^{\gamma_1} e^{i\Phi(x,y,\xi)})\, D^{\gamma_2}_{y}\chi(x-y) ] \dd y \ddd\zeta.
\end{aligned}
\end{equation*}
It follows now from \eqref{no idea what to call 1} and (\ref{eq:rs}) that
\[
\begin{aligned}
|R^{I\!\!I}_{\alpha}(t,x,\xi)| & \leq C \int_{|\zeta|>r\bra{\xi}} \int_{|x-y|<\varepsilon} (|\zeta|+\bra{\xi})^{-N} \bra{\xi}^{|\alpha|}\, \dd y \ddd\zeta & \leq C\bra{\xi}^{|\alpha|} \int_{|\zeta|>r\bra{\xi}}|\zeta|^{-N}\ddd\zeta \\ & \leq C \bra{\xi}^{|\alpha|+n-N},
\end{aligned}
\]
which yields the desired estimate when $N>|\alpha|+n$. For the derivatives of $R_{\alpha}^{I\!\!I}(t,x,\xi)$, we can get, in a similar way to the case for $R_{\alpha}^I$, an extra power of $\zeta$, which can be taken care of by choosing $N$ large and using the fact that $|x-y|<\varepsilon.$  Therefore for all multi-indices $\beta$, $\gamma\in \mathbb{Z}_{+},$
\begin{equation}
            |\partial^{\beta}_{\xi} \partial^{\gamma}_{x} R_{\alpha}^{I\!\!I}(t,x,\xi)| \leq C_{\beta, \gamma}
\end{equation}
where the constant $C_{\beta, \gamma}$ does not depend on $t$. The proof of Proposition \ref{left composition with pseudo} is now complete. \end{proof}

To prove \ref{part1} of Theorem \ref{commutator}, we apply Proposition \ref{left composition with pseudo} with $\phase = \phase_1$, $a_t = \mu$ and $R_t = P^v_t$.  Observe that the phase function $\phase_1$ belongs to the class $\Phi^2$, which implies by the mean value theorem and homogeneity in $\xi$, that there exists a constant $C_2>0$ such that
\begin{equation}\label{eq:C2}
	\abs{\nabla_x \phase_1(x,\xi)}\leq C_2 \abs{\xi}.
\end{equation}
On the other hand, the strong non-degeneracy condition on the phase can be use to show that there exists a constant $C_1>0$
\begin{equation}\label{eq:C1}
	C_1\abs{\xi}\leq\abs{\nabla_x \phase_1(x,\xi)},
\end{equation}
see for example Proposition 1.2.4 of \cite{DSFS}. The second hypothesis of Proposition \ref{left composition with pseudo}  is a direct  consequence of the $\Phi^2$ condition on the phase $\phase_1$. We should also note that $\phase_1$ is a smooth function on the support of the amplitude $a_t = \mu$.

This shows us that $P^v_tT^{\phase_1}_\mu$ is a Fourier integral operator with phase $\phase_1$ and amplitude of the form
\[
\mu(\xi)\widehat{\theta}(t\nabla_{x}\phase(x,\xi))e^{it\nabla_{x}\phase(x,\xi)\cdot v} + \sum_{0<|\alpha| < M}\frac{t^{|\alpha|}}{\alpha!}\, \sigma_{\alpha}(t,x,\xi)+t^{\mu} r(t,x,\xi).
\]
We define $W^v_{1,t}$ to be equal to the Fourier integral operator with phase $\phase_1$ and amplitude
\begin{equation} \label{s1}
- \left(\sum_{0<|\alpha| < M}\frac{t^{|\alpha|}}{\alpha!}\, \sigma_{\alpha}(t,x,\xi)+t^{\mu} r(t,x,\xi)\right)
\end{equation}
and $V^v_t$ to be equal to the Fourier integral operator with phase $\phase_1$ and amplitude
\begin{equation} \label{r}
\mu(\xi)\left(\widehat{\theta}(t\xi)e^{it\xi\cdot v} - \widehat{\theta}(t\nabla_{x}\phase(x,\xi))e^{it\nabla_{x}\phase(x,\xi)\cdot v}\right).
\end{equation}
Clearly then $[T^{\phase_1}_\mu,P^v_t] = V^v_t + W^v_{1,t}$.

By Proposition \ref{left composition with pseudo} the amplitude \eqref{s1} is in $S^0_{0,0}(n,1)$ with semi-norms which have a dependence on $t$ of the form $t^\varepsilon$ for some small $\varepsilon > 0$. Therefore, using Theorem \ref{DW226}, $\|W^v_{1,t}(f)\|_{L^2(\R^n)}\lesssim t^\varepsilon\|f\|_{L^2(\R^n)}$. To complete the proof of \ref{part1} we also need to prove \eqref{rq}, but as this is quite long, we will first dispose of \ref{part2} and \ref{part3}.

To prove \ref{part2} of Theorem \ref{commutator}, we apply Proposition \ref{left composition with pseudo} with $a_t(\xi) = \widehat{\psi}(t\xi)e^{it\xi\cdot u}\mu(\xi)$ and $\rho$ chosen to be radial, supported on an annulus and such that $\rho(t\nabla_{x}\phase(x,\xi)) = 1$ for all $\xi \in \supp(a_t)$ and all $x \in \R^n$. This is possible again since $\phase_2$ satisfies the hypotheses of Proposition \ref{left composition with pseudo} on the support of $\mu$, as discussed previously. If we set $U^u_t$ equal to the Fourier integral operator with phase $\phase_2$ and amplitude $\widehat{\psi}(t\xi)e^{it\xi\cdot u}\mu(\xi)$ then the amplitude of $R_tU^u_t$ is of the form
\begin{align*}
& \rho(t\nabla_{x}\phase(x,\xi))\widehat{\psi}(t\xi)e^{it\xi\cdot u}\mu(\xi) + \sum_{0<|\alpha| < M}\frac{t^{|\alpha|}}{\alpha!}\, \sigma_{\alpha}(t,x,\xi)+t^{\mu} r(t,x,\xi) \\
& = \widehat{\psi}(t\xi)e^{it\xi\cdot u}\mu(\xi) + \sum_{0<|\alpha| < M}\frac{t^{|\alpha|}}{\alpha!}\, \sigma_{\alpha}(t,x,\xi)+t^{\mu} r(t,x,\xi).
\end{align*}
Therefore $R_tU^u_t = T^{\phase_2}_\mu Q^u_t - W^u_{2,t}$ and $\|W^u_{2,t}(f)\|_{L^2(\R^n)}\lesssim t^\varepsilon\|f\|_{L^2(\R^n)}$ as before. Since $a_t$ is smooth and compactly supported, $\|U^u_t(f)\|_{L^2(\R^n)}\lesssim \|f\|_{L^2(\R^n)}$ with an implicit constant that grows at most polynomially in $u$ (we can see this by applying, for example, Theorem \ref{DW226}).

To prove \ref{part3} of Theorem \ref{commutator}, first observe that by the mean-value theorem we have that
\[
|m(t,y,u,v) - m(t,x,u,v)| \lesssim |y-x|
\]
with an implicit constant that is independent of $t$, $u$ and $v$. This is because $\partial^\alpha_xm(t,x,u,v)$ is bounded. Using this, we compute
\begin{align*}
|W^{u,v}_{3,t}(f)(x)| = |[R_t,M_m](f)(x)| & \lesssim \left|\int t^{-n}\check{\rho}\brkt{\frac{x-y}{t}}\big(m(t,y,u,v) - m(t,x,u,v)\big) f(y) \dd y\right| \\
& \lesssim t\int \abs{t^{-n}\inverse{\rho}\brkt{\frac{x-y}{t}}}\frac{|x-y|}{t} |f(y)| \dd y,
\end{align*}
where
\[
	\inverse{\rho}(x)=\int \rho(\xi)e^{i x.\xi}\, \ddd \xi.
\]
The estimate $\|W^{u,v}_{3,t}(f)\|_{L^2(\R^n)}\lesssim t\|f\|_{L^2(\R^n)}$ now follows from Young's inequality since the $L^1$-norm of $x \mapsto t^{-n}\inverse{\rho}(x/t)x/t$ is independent of $t$.

Now to complete the proof of Theorem \ref{commutator} it only remains to prove \eqref{rq}. Remembering that the amplitude of $V^v_t$ is given by \eqref{r}, let us define
\[
    \sigma^v(x,\xi)=\widehat{\theta}(\xi)e^{iv\cdot\xi} - \widehat{\theta}(\nabla_x \phase_1 (x,\xi))e^{iv\cdot\nabla_x \phase_1 (x,\xi)},
\]
where recall that $\widehat{\theta}\in \mathcal{C}^\infty_0(\R^n)$ is such that $\supp \widehat{\theta} \subset \set{\abs{\xi}\leq 5}$ and that $\widehat{\theta}$ is constant on
$\set{\abs{\xi}\leq 4}$.
We want to study the validity of the quadratic estimate \eqref{rq} for the Fourier integral operator
\[
    V_t^v(f)(x)=\int \mu(\xi)\sigma^v(x,t\xi) e^{i\phase_1 (x,\xi)}\widehat{f}(\xi)\ddd\xi.
\]
Observe that
\[
    \begin{split}
        \sigma^v(x,\xi)&=\brkt{\widehat{\theta}(\xi) - \widehat{\theta}(\nabla_x \phase_1 (x,\xi))}e^{iv\cdot\xi}+\widehat{\theta}(\nabla_x \phase_1 (x,\xi))\brkt{e^{iv\cdot\xi} - e^{iv\cdot\nabla_x \phase_1 (x,\xi)}}\\
        &=\sigma_I(x,\xi)e^{iv\cdot\xi}+\widehat{\theta}(\nabla_x \phase_1 (x,\xi))\sigma_{I\!\!I}^v(x,\xi),
        %\sigma_I(x,\xi)\sigma_{I\!\!I}(x,\xi)+\widehat{\theta}(\xi)\sigma_{I\!\!I}(x,\xi).
%        &=\sigma_{I\!\!I}^v(x,\xi).
    \end{split}
\]
where
\begin{align*}
\sigma_I(x,\xi) & = \widehat{\theta}(\xi) - \widehat{\theta}(\nabla_x \phase_1 (x,\xi))\quad {\rm and}\\
\sigma_{I\!\!I}^v(x,\xi) & = e^{iv\cdot\xi} - e^{iv\cdot\nabla_x \phase_1 (x,\xi)}.
\end{align*}
Thus
\[
	V_t^vf=V^v_{t,I}  f+V^v_{t,I\!\!I}f,
\]
where $V^v_{t,I}$ and $V^v_{t,I\!\!I}$ are the Fourier integral operators with amplitudes $\mu(\xi)\sigma_I(x,t\xi)e^{i tv\xi}$ and $\widehat{\theta}(\nabla_x \phase_1 (x,\xi))\sigma_{I\!\!I}^v(x,\xi)\mu(\xi)$ respectively. Thus, it suffices to prove the desired quadratic estimate for $V^v_{t,I}$ and $V^v_{t,I\!\!I}$ separately. This will finally be achieved in Propositions \ref{prop:1} and \ref{prop:2} below, but to help us in this task we shall first prove some technical lemmas.
%\subsubsection{\bf Dealing with the first term}

\begin{lem}\label{lemaA} For multi-indices ${\abs{\alpha}}+{\abs{\beta}}\geq 1$ and for any $N_1,\ldots, N_{\abs{\alpha}+\abs{\beta}}\in [0,\infty)$,
\begin{equation*}\label{eq:0}
	\abs{\d^\alpha_\xi \d^\beta_x \widehat{\theta}(t\nabla_x\phase_1 (x,\xi))}\lesssim \sum_{j=1}^{\abs{\alpha}+\abs{\beta}} t^j \abs{\xi}^{j-\abs{\alpha}} \p{t\xi}^{-N_j}.
\end{equation*}
\end{lem}
\begin{proof} Let $d=\abs{\alpha}+\abs{\beta}$ be the order of derivatives. Observe that for any $j=1,\ldots,n$,
\begin{align*}
\abs{\d_{\xi_j} \widehat{\theta}(t\nabla_x\phase_1 (x,\xi))} & =t\abs{\p{\nabla\widehat{\theta}(t\nabla_x\phase_1 (x,\xi)), \d_{\xi_j}\nabla_x\phase_1 (x,\xi)}} \\
& \lesssim t\p{t\nabla_x\phase_1 (x,\xi)}^{-N_1}\lesssim t\p{t\xi}^{-N_1}.
\end{align*}
Similarly, we have
\[
\abs{\d_{x_j} \widehat{\theta}(t\nabla_x\phase_1 (xr,\xi))}=t\abs{\p{\nabla\widehat{\theta}(t\nabla_x\phase_1 (x,\xi)), \d_{x_j}\nabla_x\phase_1 (x,\xi)}}\lesssim t\p{t\xi}^{-N_1} \abs{\xi}.
\]
So the inequality holds for multi-indices with order $d=1$. Assume that it holds for multi-indices of order smaller than $d$. Let us prove that under this assumption it holds for those of order $d$ too, thus by induction, the lemma will be proved.

Assume that $\abs{\alpha}\geq 1$ and let $\alpha_j\neq 0$. Let $\tilde{\alpha}$ the multi-index that $\tilde{\alpha}_k=\alpha_k$ for $j\neq k$ and $\tilde{\alpha}_j=\alpha_j-1$. Then,
\[
	\begin{split}
	\d^\alpha_\xi &\d^\beta_x \widehat{\theta}(t\nabla_x\phase_1 (x,\xi))=t{\d^{\tilde{\alpha}}_\xi \d^\beta_x {\p{\nabla\widehat{\theta}(t\nabla_x\phase_1 (x,\xi)), \d_{\xi_j}\nabla_x\phase_1 (x,\xi)}}}\\
		&=t\sum_k \sum_{\alpha_1+\alpha_2=\tilde{\alpha}}\sum_{\beta_1+\beta_2=\beta} C_{\tilde{\alpha},\beta} \brkt{\d^{\alpha_1}_\xi\d^{\beta_1}_x
		(\d_k\widehat{\theta})(t\nabla_x\phase_1 (x,\xi))} \brkt{\d^{\alpha_2}_\xi\d^{\beta_2}_x
		(\d_{\xi_j,x_k}^2\phase_1 (x,\xi))}
	\end{split}
\]
Then
\[
	\begin{split}
	\abs{	\d^\alpha_\xi \d^\beta_x \widehat{\theta}(t\nabla_x\phase_1 (x,\xi))}
		&\lesssim t\sum_{\alpha_1+\alpha_2=\tilde{\alpha}}\sum_{\beta_1+\beta_2=\beta} C_{\tilde{\alpha},\beta} \sum_{j=1}^{\abs{\alpha_1}+\abs{\beta_1}} t^j \abs{\xi}^{j-\abs{\alpha_1}} \p{t\xi}^{-N_j}\abs{\xi}^{-\abs{\alpha_2}}+\\
		&\qquad +t(\d_k\widehat{\theta})(t\nabla_x\phase_1 (x,\xi)) \brkt{\d^{\tilde{\alpha}}_\xi \d^{\beta}_x  (\d_{\xi_j,x_k}^2\phase_1 (x,\xi))}\\
		&\lesssim \sum_{j=1}^{\abs{\tilde{\alpha}}+\abs{\beta}} t^{j+1} \abs{\xi}^{j-\abs{\tilde{\alpha}}} \p{t\xi}^{-N_j}+t\abs{\xi}^{1-\abs{\alpha}}\p{t\xi}^{-N_0}\\
		&\lesssim \sum_{j=2}^{\abs{\alpha}+\abs{\beta}} t^{j} \abs{\xi}^{j-\abs{\alpha}} \p{t\xi}^{-N_{j-1}}+t\abs{\xi}^{1-\abs{\alpha}}\p{t\xi}^{-N_0}.
	\end{split}
\]
For $\abs{\alpha}=0$, the result follows from a similar argument. We omit the details.
\end{proof}

Let $C_1$ and $C_2$ be the constants appearing in \eqref{eq:C1} and \eqref{eq:C2}. We are free to assume that $C_1\leq 1\leq C_2$.

\begin{lem}\label{lem:1} The following statements hold:
\begin{enumerate}[label={\upshape (\roman*)}]
	\item The function $\widehat{\theta}(\nabla_x\phase_1 (x,\xi))$ is supported in $\R^n\times \set{\abs{\xi}\leq \frac{5}{C_1}}$ and it is constant on $\R^n\times\set{\abs{\xi}\leq \frac{4}{C_2}}$.
	\item There exist $\widehat{\psi}\in \mathcal{C}^\infty_0(\R^n)$, $0<r'<4$ and $5<R'<\infty$ satisfying that $\supp \widehat{\psi} \subset \set{\abs{\xi}\leq R'}$, $\widehat{\psi}=0$ on $\set{\abs{\xi}\leq r'}$ and
\begin{equation}\label{eq:1}
    \sigma_I(x,\xi)=\sigma_I(x,\xi)\widehat{\psi}(\xi).
\end{equation}
\end{enumerate}
\begin{proof}Observe that $\frac{4}{C_2}\leq 4 \leq 5\leq \frac{5}{C_1}$. If $\abs{\xi}>\frac{5}{C_1}$, then $\abs{\nabla_x \phase_1 (x,\xi)}\geq  C_1\abs{\xi}>5$, which yields that for any $x\in \R^n$, $\widehat{\theta}(\nabla_x\phase_1 (x,\xi))=\sigma_I(x,\xi)=0$. On the other hand, if $\abs{\xi}<\frac{4}{C_2}$, then $\abs{\nabla_x \phase_1 (x,\xi)}\leq C_2 \abs{\xi}<4$, which yields that for any $x\in\R^n$, $\widehat{\theta}(\nabla_x(x,\xi))$ is constant and $\sigma_I(x,\xi)=0$.

The last assertion follows by taking $\widehat{\psi}\in\mathcal{C}^\infty_0(\R^n)$ such that is equal to one on the set $\set{\frac{4}{C_2}\leq \abs{\xi}\leq \frac{5}{C_1}}$ and it is equal to $0$ on $\set{\abs{\xi}\leq r'}$ with $r'<\frac{4}{C_2}$.
\end{proof}
\end{lem}

Define $\sigma_t(x,\xi)=\sigma_I(x,t\xi)$ and let $T_{\sigma_t}$ be the Fourier integral operator with amplitude $\sigma_t$ and phase function $\phase_1 (x,\xi)$.
\begin{lem} If $t\leq 1$, $\sigma_t\in S^0_{1,0}(n,1)$ uniformly on $t$. That is, for any multi-indices $\alpha,\beta$
\begin{equation}\label{eq:3}
    \sup_{0<t<1}\sup_x \abs{\d^\beta_x \d^\alpha_\xi \sigma_t(x,\xi)}\lesssim \p{\xi}^{-\abs{\alpha}}.
\end{equation}
As a consequence
\begin{equation}\label{eq:2}
    \sup_{0<t<1}\norm{T_{\sigma_t}}_{L^2\to L^2}=\mathfrak{c}<+\infty.
\end{equation}
\end{lem}
\begin{proof} Suppose first that we have shown that $\sigma_I\in S^0_{1,0}(n,1)$. Observe that, for $t<1$, $\p {t\xi}\geq t\p{\xi}$. Then,
\[
    \abs{\d^\beta_x \d^\alpha_\xi \sigma_t(x,\xi)}=\abs{t^{\abs{\alpha}}(\d^\beta_x \d^\alpha_\xi \sigma_I) (x,t\xi)}
    \leq C t^{\abs{\alpha}}\p{t\xi}^{-\abs{\alpha}}\leq \p{\xi}^{-\abs{\alpha}}.
\]

To prove that $\sigma_I\in S^0_{1,0}(n,1)$, it is sufficient to prove that $\widehat{\theta}(\nabla_x \phase_1 (x,\xi))\widehat{\psi}(\xi)\in S^0_{1,0}(n,1)$ with $\psi$ as in \eqref{eq:1}, since we can then see $\widehat{\theta}(\xi)\widehat{\psi}(\xi)\in S^0_{1,0}(n,1)$ as a special case of this. The result will follow if we prove that, for any pair of multi-indices $\alpha,\beta$,
\[
    \sup_{r'\leq \abs{\xi}\leq R'} \p{\xi}^{\abs{\alpha}}\sup_x \abs{\d^\beta_x \d^\alpha_\xi \brkt{\widehat{\theta}(\nabla_x \phase_1 (x,\xi))}}<+\infty.
\]
But Lemma \ref{lemaA} yields that, for $N\geq \abs{\alpha}$,
\[
    \abs{\d^\beta_x \d^\alpha_\xi \brkt{\widehat{\theta}(\nabla_x \phase_1 (x,\xi))}}%\leq C_{N,\alpha,\beta} H_{\alpha,\beta}(r',R') \p{\nabla_x \phase_1 (x,\xi)}^{-N}
    \leq C_{N,\alpha,\beta} H_{\alpha,\beta}(r',R') \p{\xi}^{-N},
\]
where
\[
    H_{\alpha,\beta}(r,R)=\sum_{j=1}^{\abs{\alpha}} r^{1-j}+\sum_{j=1}^{\abs{\beta}}R^j.
\]

The last assertion of the lemma follows from \eqref{eq:3} and Theorem \ref{DW226}.
\end{proof}

\begin{prop}\label{prop:1} For any $f\in L^2$,
\[
    \sup_{v\in \R^n}\brkt{\int_0^1 \norm{V^v_{t,I} f}^2_{L^2}\, \frac{\dd t}{t}}^{\frac 1 2}\lesssim \norm{f}_{L^2}.
\]
\end{prop}
\begin{proof} Observe that by Lemma \ref{lem:1}, $T_{\sigma_t} \brkt{\tau_{tv}f}(x)=T_{\sigma_t} \brkt{\psi_t^v*f}(x)$, where $\widehat{\psi_t^v}(\xi)=\widehat{\psi}(t\xi)e^{it v\xi}$ and  $\tau_{tv}f(x)=f(x-tu)$. Then, \eqref{eq:2} and the properties on $\psi$ yields
\[
    \begin{split}
    \int_0^1\int_{\R^n} \abs{T_{\sigma_t} \brkt{\tau_{tv}f}(x)}^2\, \dd x \frac{\dd t}{t}&\leq
    \mathfrak{c}\int_0^1 \int_{\R^n} \abs{\psi_t^v*f(x)}^2\, \dd x \frac{\dd t}{t}\\
    &\leq \mathfrak{c}\int_0^\infty \int_{\R^n} \abs{\psi_t^v*f(x)}^2\, \dd x \frac{\dd t}{t}\lesssim \mathfrak{c}\norm{f}_{L^2}^2.
    \end{split}
\]
Finally observe that  $V^v_{t,I}f=T_{\sigma_t}  \brkt{\tau_{tv} Sf}$ where $Sf(x)=\int \mu(\xi) \widehat{f}(\xi) e^{ix \xi}\ddd\xi$, which is a bounded operator on $L^2$, so the proposition is proved.
\end{proof}

%\subsubsection{\bf Dealing with the Second term}

We now turn our attention to $V_{t,I\!\!I}^v$. We will make use of the following lemmas.
\begin{lem}\label{lemaB} For any multi-indices $\alpha$ and $\beta$,
\[
	\abs{\d^\alpha_\xi \d^\beta_x \sigma_{I\!\!I}^v(x,\xi)}\lesssim \sum_{j=1}^{\max\brkt{\abs{\alpha}+\abs{\beta},1}} \abs{v}^j \abs{\xi}^{j-\abs{\alpha}}.
\]
\end{lem}
\begin{proof} For $\alpha=\beta=0$,
\[
	\abs{e^{iv\cdot\xi}-e^{iv\cdot\nabla_x \phase_1 (x,\xi)}}\leq \abs{v}\abs{\xi-\nabla_x \phase_1 (x,\xi)}\lesssim  \abs{v}\abs{\xi}.
\]
With a similar argument as that of the proof of Lemma \ref{lemaA}, we prove that, for $\abs{\alpha}+\abs{\beta}\geq 1$,
\begin{equation*}\label{eq:10}
	\abs{\d^\alpha_\xi \d^\beta_x e^{i v\cdot \nabla_x\phase_1 (x,\xi)}}\lesssim \sum_{j=1}^{\abs{\alpha}+\abs{\beta}}  \abs{v}^j \abs{\xi}^{j-\abs{\alpha}}.
\end{equation*}
A similar estimate holds for $e^{iv.\xi}$. From these estimates, the lemma follows.
\end{proof}

\begin{lem}\label{lem1}
Let $t\leq 1$.  Let
\[
    a_{t}(x,\xi)=\mu(\xi)\widehat{\theta}(t\nabla_x\phase_1 (x,\xi))\sigma^{v}_{I\!\!I}(x,t\xi).
\]
Then $a_t \in S^0_{0,0}(n,1)$ and
\begin{equation}\label{eq:key1}
    \sup_{0<t<1}\abs{\d^\alpha_x\d^\beta_\xi a_{t}(x,\xi)}\lesssim P_{\alpha,\beta}(\abs{tv}\abs{\xi}).
\end{equation}
where $P_{\alpha,\beta}(r)= \sum_{j=1}^{\min(\abs{\beta}+\abs{\alpha},1)}r^j$.
\end{lem}

\begin{proof}Observe that $\sigma_{I\!\!I}^v(x,t\xi)=\sigma_{I\!\!I}^{tv}(x,\xi)$ and also that $a_{t}$ is supported in
 \[
    D_{t}=\set{\abs{\xi}\leq \frac{5}{t C_1}}\cap \set{\abs{\xi}\geq 1}.
 \]
Lemma \ref{lemaA} and Lemma \ref{lemaB} yield that for any $\xi\in D_{t}$,
\[
    \abs{\d^\alpha_\xi \d^\beta_x \theta(t\nabla_x\phase_1 (x,\xi))}\lesssim \sum_{j=1}^{\abs{\beta}+\abs{\alpha}} \abs{\xi}^{-\abs{\alpha}}
    \abs{t\xi}^j\lesssim 1,
\]
and
\[
    \abs{\d^\alpha_\xi \d^\beta_x \sigma_{I\!\!I}^{tv}(x,\xi)}\lesssim \sum_{j=1}^{\min(\abs{\beta}+\abs{\alpha},1)}\brkt{\abs{tv}\abs{\xi}}^j. %\lesssim \sum_{j=1}^{\min(\abs{\beta}+\abs{\alpha},1)}{\abs{v}}^j.
\]
Thus, Leibniz's formula yields \eqref{eq:key1}.
\end{proof}

\begin{lem} \label{lem2} Let $t\leq 1$ and let $0<s<\infty$. Let $\psi\in \mathcal{C}^\infty(\R^n)$ such that $\widehat{\psi}$ is supported in $\set{\frac{5}{C_1}\leq\abs{\xi}\leq \frac{20}{C_1}}$ and such that $\int_0^\infty |\widehat{\psi}(t\xi)|^2 \frac{\dd t}{t} = 1$ for $\xi \neq 0$. Consider
\[
    a_{s,t}(x,\xi)=\mu(\xi)\widehat{\psi}(s\xi)\widehat{\theta}(t\nabla_x\phase_1 (x,\xi))\sigma^{v}_{I\!\!I}(x,t\xi).
\]
Then
\begin{equation}\label{eq:key}
    \abs{\d^\alpha_x\d^\beta_\xi a_{s,t}(x,\xi)}\lesssim P_{\alpha,\beta}(\abs{v}) \min\brkt{\frac{s}{t},\frac{t}{s}}.
\end{equation}
where $P_{\alpha,\beta}(r)= \sum_{j=1}^{\min(\abs{\beta}+\abs{\alpha},1)}r^j$.
\end{lem}
\begin{proof}
Observe that $a_{s,t}$ is supported in
 \[
    D_{s,t}=\set{\frac{5}{s C_1}\leq\abs{\xi}\leq \frac{20}{s C_1}}\cap \set{\abs{\xi}\leq \frac{5}{t C_1}}\cap \set{\abs{\xi}\geq 1}.
 \]
Then if $t\geq s$ then $\widehat{\psi}(s\xi)\widehat{\theta}(t\nabla_x\phase_1 (x,\xi))=0$ and then, \eqref{eq:key} trivially holds. For $s>t$ and $\xi \in D_{s,t}$,
\[
    \abs{\d^\alpha_\xi \d^\beta_x \sigma^{tv}(x,\xi)}\lesssim \sum_{j=1}^{\min(\abs{\beta}+\abs{\alpha},1)}\brkt{\abs{tv}\abs{\xi}}^j\lesssim \sum_{j=1}^{\min(\abs{\beta}+\abs{\alpha},1)}\brkt{\frac{t\abs{v}}{s}}^j
\lesssim \frac{t}{s}\sum_{j=1}^{\min(\abs{\beta}+\abs{\alpha},1)}{\abs{v}}^j,
\]
by Lemma \ref{lem1}.
\end{proof}

\begin{prop}\label{prop:2} Let $\{V^v_{t,I\!\!I}\}_{0<t\leq 1,v\in \R^n}$ be the family of operators defined by
\[
	V^v_{t,I\!\!I}f(x)=\int \mu(\xi)\widehat{\theta}(t\nabla_x\phase_1 (x,\xi))\sigma^{tv}_{I\!\!I}(x,\xi)\widehat{f}(\xi)e^{i\phase_1 (x,\xi)} \ddd\xi.
\]
There is a polynomial $\mathcal{P}$ such that for any $v\in\R^n$ and any $f\in L^2$,
\[
    \brkt{\int_0^1 \norm{V^v_{t,I\!\!I} f}^{2}_{L^2}\frac{\dd t}{t}}^{\frac 1 2}\lesssim \mathcal{P}(\abs{v}) \norm{f}_{L^2}.
\]
\end{prop}
\begin{proof} Theorem \ref{DW226}, Lemma \ref{lem1} and Lemma \ref{lem2} yield that there exist two polynomials $\mathcal{P}_1$ and $\mathcal{P}_2$ such that, for any $v\in\R^n$,
\[
	\sup_{0<t\leq 1}\norm{V^v_{t,I\!\!I}}_{L^2\to L^2}\lesssim \mathcal{P}_1(\abs{v})
\]
and, for $0<t\leq 1$, $0<s<\infty$,
\[
	\norm{V^v_{t,I\!\!I}Q_s}_{L^2\to L^2}\lesssim \mathcal{P}_2(\abs{v}) \min\brkt{\frac{s}{t},\frac{t}{s}},
\]
where $Q_s$ denotes the convolution operator with kernel $\psi_s(x)=s^{-n}\psi(x/s)$, where $\psi$ satisfies the conditions in Lemma \ref{lem2}. Hence, defining $V^v_{t,I\!\!I}=0$ for $t>1$, we can apply Corollary 8.6.4 in \cite{G} to conclude that
\[
    \brkt{\int_0^1 \norm{V^v_{t,I\!\!I} f}^{2}_{L^2}\frac{\dd t}{t}}^{\frac 1 2}\lesssim \brkt{\mathcal{P}_1(\abs{v})+\mathcal{P}_2(\abs{v})} \norm{f}_{L^2}.
\]
\end{proof}

\end{document}